\documentclass[12pt]{amsart}

\usepackage{amssymb}
\usepackage{amsmath}
\usepackage{amsthm}
\usepackage{amscd}

\usepackage{color}

\definecolor{darkgreen}{rgb}{0,0.5,0}

\def\qed{\hfill$\square$}
\numberwithin{equation}{section}
\newtheorem{thm}[equation]{\sc Theorem}
\newtheorem{lem}[equation]{\sc Lemma}
\newtheorem{cor}[equation]{\sc Corollary}
\newtheorem{prop}[equation]{\sc Proposition}

\newtheoremstyle{notation}{3pt}{3pt}{}{}{\itshape}{:}{.5em}{\thmname{#1}}
\theoremstyle{notation}

\newtheorem{rem}{\it Remark}

\newtheorem{defin}{\it Definition}
\newtheorem{ex}{\it Example}

\newtheorem{dedication}{\it Dedication}
\newtheorem{acknowledgement}{\it Acknowledgement}

\def\Cok{\mbox{\rm Cok}}

\def\mod{\mbox{{\rm mod}}}
\def\Hom{\mbox{\rm Hom}}
\def\End{\mbox{\rm End}}

\def\rad{\mbox{\rm rad}}
\def\soc{\mbox{\rm soc\,}}

\input prepictex   \input pictex    
\input postpictex

\newcounter{boxsize}
\newcounter{tempcounter}
\newcommand{\smallentryformat}{\scriptstyle\sf}

\newcommand\smbox{\put(0,0){\line(1,0){\value{boxsize}}}%
  \put(\value{boxsize},0){\line(0,1){\value{boxsize}}}%
  \put(0,0){\line(0,1){\value{boxsize}}}%
  \put(0,\value{boxsize}){\line(1,0){\value{boxsize}}}}
\newcommand\numbox[1]{\put(0,0)\smbox%
  \put(0,0){\makebox(\value{boxsize},\value{boxsize})[c]{%
      $\smallentryformat#1$}}}
\newcommand\singlebox[1]{\raisebox{-.4ex}{\begin{picture}(4,0)\setcounter{boxsize}{3}%
    \put(0,0)\smbox%
    \put(0,0){\makebox(\value{boxsize},\value{boxsize})[c]{%
      $\scriptstyle\sf#1$}}\end{picture}}}
\newcommand\boxwithrow[2]{\raisebox{-.4ex}{\begin{picture}(5,0)\setcounter{boxsize}{3}%
    \put(0,0)\smbox%
    \put(0,0){\makebox(\value{boxsize},\value{boxsize})[c]{%
      $\scriptstyle\sf#1$}}
    \put(\value{boxsize},1.5){\line(1,0){1}}
    \put(\value{boxsize},1){$\;\scriptscriptstyle\sf#2$}
    \end{picture}}}

\def\sbullet{\makebox(0,0){$\scriptstyle\bullet$}}
\newcommand\boxes[2]{\ifthenelse{#2=3}{$\scriptstyle P_2^{#1}$}{%
                                       $\scriptstyle P_{#2}^{#1}$}}

\newcommand\mylabel[1]{\label{#1}}%\begin{purple}$\to$#1\end{purple}}

\begin{document}

\begin{center}
  \Large
  Box moves on Littlewood-Richardson tableaux\\[.3ex]
and an application to invariant subspace varieties

  \normalsize
  \bigskip Justyna Kosakowska\footnote{The first named author
    is partially supported 
    by Research Grant No.\ DEC-2011/02/A/ ST1/00216
    of the Polish National Science Center.}

  \footnotesize
  \smallskip{\it Faculty of Mathematics and Computer Science,
              Nicolaus Copernicus University\\
              ul.\ Chopina 12/18, 87-100 Toru\'n, Poland}\\
           {\tt justus@mat.umk.pl}              

  \normalsize         
  \medskip Markus Schmidmeier\footnote{This research is partially supported 
    by a~Travel and Collaboration Grant from the Simons Foundation
    (Grant number 245848 to the second named author).}

  \smallskip\footnotesize     {\it  Department of Mathematical Sciences, 
              Florida Atlantic University\\
              777 Glades Road,
              Boca Raton, Florida 33431}\\
     {\tt markus@math.fau.edu}
\end{center}

\medskip
{\bf Abstract:} 
    In his 1951 book
    {\it ``Infinite Abelian Groups'',} Kaplansky gives a combinatorial
    characterization of the isomorphism types of embeddings of a 
    cyclic subgroup in a finite abelian group.
    In this paper we first use partial maps on Littlewood-Richardson tableaux
    to generalize this result to finite direct sums of such embeddings. 
    We then focus on an application to
    invariant subspaces of nilpotent linear operators.
  We develop a criterion to decide if two irreducible components in
  the representation space are in the boundary partial order.



%==============================================================================
\section{Introduction}
%==============================================================================

Let $\Lambda$ be a discrete valuation ring with maximal ideal generator $p$
and radical factor field $k$.

 We study the category of all embeddings $(A\subset B)$ 
of a submodule $A$ in a finite length $\Lambda$-module $B$.
Two examples are of particular interest:
If $\Lambda=\widehat{\mathbb Z_{(p)}}$ is the completion of the ring
of $p$-adic integers, then we are dealing with embeddings of a subgroup
in a finite abelian $p$-group.
Their classification, up to isomorphy,
is the well-known Birkhoff Problem \cite{birkhoff}.
On the other hand, if $\Lambda=k[[T]]$ is the power series ring,
then an embedding $(A\subset B)$ consists of a nilpotent linear operator $B$
and an invariant subspace $A$.

In general, any
 finite length $\Lambda$-module $A$ is isomorphic to a~direct sum 
 $$A\cong \Lambda/(p^{\alpha_1})\oplus \Lambda/(p^{\alpha_2})\oplus \ldots\oplus \Lambda/(p^{\alpha_n}),$$
 where $\alpha_1\geq \alpha_2\geq\ldots\geq \alpha_n$. 
Therefore, there is a~bijection ($A\mapsto \alpha=(\alpha_1,\ldots,\alpha_n)$)
 between the set of isomorphism classes of finite length $\Lambda$-modules 
and the set of partitions. 
 The partition $\alpha$ is the {\it type} of $A$
and we write $A\cong N_\alpha(\Lambda)=N_\alpha$.

 It is natural to associate with an~embedding $(A\subset B)$ of finite length $\Lambda$-modules
 the~triple of partitions $(\alpha,\beta,\gamma)$,
 where $\alpha$, $\beta$, $\gamma$ are the types 
of $A$, $B$ and $C=B/A$, respectively.

\medskip
We call an embedding $(A\subset B)$ {\it cyclic} provided 
the submodule $A$ is cyclic
as a~$\Lambda$-module,
that is, $A$ is either indecomposable or zero.
Cyclic embeddings have been classified, up to isomorphy, by 
Kaplansky \cite[Theorem~25]{kap} in terms of the ``height
sequence'' or ``Ulm sequence'' of the submodule generator.
The first aim in this paper is to derive a simple combinatorial description
of the isomorphism types of direct sums of cyclic embeddings.
They are given in terms of partial maps on Littlewood-Richardson tableaux
(Theorem~\ref{theorem-combinatorial-classification}).

\medskip
The main goal
of the paper is to shed light on the geometry of the 
representation space of invariant subspace varieties.
Suppose that $\Lambda$ is the power series ring $k[[T]]$ with 
coefficients in an algebraically closed field $k$.
The embeddings corresponding to partition type $(\alpha,\beta,\gamma)$
form a constructible subset 
$\mathbb V_{\alpha,\gamma}^\beta=\{f:N_\alpha\hookrightarrow N_\beta: \Cok f\cong N_\gamma\}$ 
of the affine variety
of all $k$-linear homomorphisms $\Hom_k(A,B)$.
  By the Theorem of Green and Klein \cite{klein1},
  the variety is non-empty if and only if there exists a Littlewood-Richardson
  tableau (LR-tableau) of type $(\alpha,\beta,\gamma)$.
  More precisely, we can assign to each embedding
  a tableau; by $\mathbb V_\Gamma$ we denote the
  subset of $\mathbb V_{\alpha,\gamma}^\beta$ of all embeddings
  with tableau $\Gamma$. Then

$$\mathbb V_{\alpha,\gamma}^\beta=\bigcup^\bullet_{\Gamma}\mathbb{V}_\Gamma, $$
where the union is
indexed by the LR-tableaux of type $(\alpha,\beta,\gamma)$.
The closures (in the Zariski topology)
$\overline{\mathbb{V}}_\Gamma$ form the irreducible components of $\mathbb V_{\alpha,\gamma}^\beta$,
see Section \ref{sec-tableaux} and \cite{ks-mfo} for details.

We are interested in how the irreducible components are linked
in the representation space $\mathbb V_{\alpha,\gamma}^\beta$.
We illustrate this in an example.

\begin{ex}
For $\alpha=(3,2)$, $\beta=(5,4,3,2,1)$ and $\gamma=(4,3,2,1)$,
we consider embeddings $(A\subset B)$ where $A\cong N_\alpha=\Lambda/(p^3)\oplus
\Lambda/(p^2)$, $B=N_\beta$, and where the cokernel $B/A$ 
is isomorphic to $N_\gamma$.
Associated with each embedding is an LR-tableau (see Section \ref{sec-tableaux}); 
we denote by $\mathbb V_\Gamma$
the collection of all embeddings in the affine variety 
$\Hom_k(N_\alpha,N_\beta)$ which have the corresponding tableau $\Gamma$.
There are five LR-tableaux of type $(\alpha,\beta,\gamma)$:

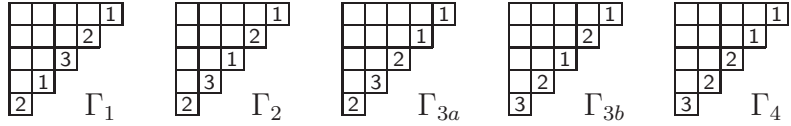
$$
\begin{picture}(18,15)
\multiput(0,12)(3,0)4{\smbox}
\put(12,12){\numbox{1}}
\multiput(0,9)(3,0)3{\smbox}
\put(9,9){\numbox{2}}
\multiput(0,6)(3,0)2{\smbox}
\put(6,6){\numbox{3}}
\put(0,3){\smbox}
\put(3,3){\numbox{1}}
\put(0,0){\numbox{2}}
\put(10,0){$\Gamma_1$}
\end{picture}
\quad
\begin{picture}(18,15)
\multiput(0,12)(3,0)4{\smbox}
\put(12,12){\numbox{1}}
\multiput(0,9)(3,0)3{\smbox}
\put(9,9){\numbox{2}}
\multiput(0,6)(3,0)2{\smbox}
\put(6,6){\numbox{1}}
\put(0,3){\smbox}
\put(3,3){\numbox{3}}
\put(0,0){\numbox{2}}
\put(10,0){$\Gamma_{2}$}
\end{picture}
\quad
\begin{picture}(18,15)
\multiput(0,12)(3,0)4{\smbox}
\put(12,12){\numbox{1}}
\multiput(0,9)(3,0)3{\smbox}
\put(9,9){\numbox{1}}
\multiput(0,6)(3,0)2{\smbox}
\put(6,6){\numbox{2}}
\put(0,3){\smbox}
\put(3,3){\numbox{3}}
\put(0,0){\numbox{2}}
\put(10,0){$\Gamma_{3a}$}
\end{picture}
\quad
\begin{picture}(18,15)
\multiput(0,12)(3,0)4{\smbox}
\put(12,12){\numbox{1}}
\multiput(0,9)(3,0)3{\smbox}
\put(9,9){\numbox{2}}
\multiput(0,6)(3,0)2{\smbox}
\put(6,6){\numbox{1}}
\put(0,3){\smbox}
\put(3,3){\numbox{2}}
\put(0,0){\numbox{3}}
\put(10,0){$\Gamma_{3b}$}
\end{picture}
\quad
\begin{picture}(18,15)
\multiput(0,12)(3,0)4{\smbox}
\put(12,12){\numbox{1}}
\multiput(0,9)(3,0)3{\smbox}
\put(9,9){\numbox{1}}
\multiput(0,6)(3,0)2{\smbox}
\put(6,6){\numbox{2}}
\put(0,3){\smbox}
\put(3,3){\numbox{2}}
\put(0,0){\numbox{3}}
\put(10,0){$\Gamma_4$}
\end{picture}
$$

Hence the space of all embeddings corresponding to partition type $(\alpha,\beta,\gamma)$,
$$\mathbb V_{\alpha,\gamma}^\beta=\{f:N_\alpha\hookrightarrow N_\beta: \Cok f\cong N_\gamma\}
\subset \Hom_k(N_\alpha,N_\beta),$$
has five irreducible components, namely the 
closures $\overline{\mathbb V}_{\Gamma_x}$ where $x$ is 
one of the symbols in $\{1, 2, 3a, 3b, 4\}$.  

\medskip
We are interested in the boundary of the irreducible components. 
Write $$\Gamma\preceq_{\sf boundary}\widetilde\Gamma\quad\text{if}\quad
\mathbb V_{\widetilde\Gamma}\cap \overline{\mathbb V}_\Gamma\neq \emptyset.$$
Note that $\Gamma\preceq_{\sf boundary}\widetilde\Gamma$ and 
$\widetilde\Gamma\preceq_{\sf boundary}\Gamma$ cannot both hold
unless $\Gamma=\widetilde\Gamma$, see \cite{ks-mfo}.
It turns out that
the {\it boundary relation,} defined to be the  
transitive closure $\leq_{\sf boundary}$ of $\preceq_{\sf boundary}$,
is a geometrically motivated partial order on 
the set of LR-tableaux of type $(\alpha,\beta,\gamma)$.
We denote this partially ordered set by $\mathcal P_{\sf boundary}$.

\medskip
In this paper we provide a~tool to determine 
this poset.
\medskip

  We observe in Lemma~\ref{lemma-union-columns} that a
  LR-tableau can be realized by a direct sum
  of cyclic embeddings if and only if it is a union of columns
  with subsequent entries.
Suppose  $\Gamma$ and $\widetilde\Gamma$ are two LR-tableaux of the same type,
both unions of columns with subsequent entries, 
such that they differ in exactly two columns.
If the smallest entry occurs in a higher position in $\widetilde\Gamma$ then
we say that $\widetilde\Gamma$ {\it is obtained from} $\Gamma$ {\it by an
increasing box move,} see Section \ref{sec-box} for details.
\end{ex}

\medskip
In this situation we have the following result.

\begin{thm}\mylabel{theorem-box-boundary}
If $\widetilde\Gamma$ is obtained from $\Gamma$
by an increasing box move, then $\Gamma\prec_{\sf boundary}\widetilde\Gamma$.
\end{thm}

In our example, any tableau is a~union of columns with subsequent entries.
Consider the first two tableaux, $\Gamma_1$ and $\Gamma_2$,
they differ in the second and third column.  More precisely, $\Gamma_2$ is obtained
from $\Gamma_1$ by exchanging the contents of those two columns in such a way that
the smallest entry is in a higher position in $\Gamma_2$, this is, $\Gamma_2$ is obtained from $\Gamma_1$ by an increasing
box move. Hence $\Gamma_1\prec_{\sf boundary}\Gamma_2$.
Further box moves yield relations corresponding to 
the remaining edges of the Hasse diagram of the poset $\mathcal P_{\sf boundary}$.

$$
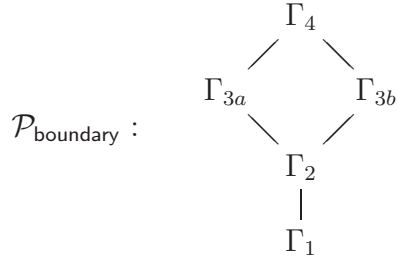
\begin{picture}(20,30)
\put(-20,15){\makebox(0,0){$\mathcal P_{\sf boundary}:$}}
\put(10,0){\makebox(0,0){$\Gamma_1$}}
\put(10,10){\makebox(0,0){$\Gamma_{2}$}}
\put(0,20){\makebox(0,0){$\Gamma_{3a}$}}
\put(20,20){\makebox(0,0){$\Gamma_{3b}$}}
\put(10,30){\makebox(0,0){$\Gamma_4$}}
\put(13,13){\line(1,1){4}}
\put(7,13){\line(-1,1){4}}
\put(3,23){\line(1,1){4}}
\put(17,23){\line(-1,1){4}}
\put(10,3){\line(0,1){4}}
\end{picture}
$$

In general, the poset $\mathcal P_{\sf boundary}$ may be difficult to determine.  
If, however, the shape $\beta\setminus\gamma$ of the LR-tableaux is a horizontal
and vertical strip  then it turns out that any two tableaux in dominance partial
order can be transformed into each other by using increasing box moves
\cite{kst}, see Section \ref{sec-tableaux} for definitions.

\medskip
Hence we obtain:

\begin{cor}\mylabel{cor-rook-strip}
Suppose $\alpha$, $\beta$, $\gamma$ are partitions such that $\beta\setminus\gamma$
is a horizontal and vertical strip.
The following statements are equivalent for LR-tableaux $\Gamma$, $\widetilde\Gamma$
of type $(\alpha,\beta,\gamma)$.
\begin{enumerate}
\item $\widetilde\Gamma$ is obtained from $\Gamma$ by a sequence of increasing
  box moves.
\item $\Gamma\leq_{\sf boundary}\widetilde \Gamma$.
\item $\Gamma$ and $\widetilde\Gamma$ are in the dominance partial order.
\end{enumerate}
\end{cor}

Organization of the paper:

\medskip
First, in Section~\ref{sec-poles}, we give a combinatorial 
characterization of direct sums of cyclic embeddings
(Theorem~\ref{theorem-combinatorial-classification}),
generalizing the corresponding result by Kaplansky for poles.
The characterization is given in terms of partial maps on
LR-tableaux which are introduced in (\ref{sec-tableaux}).

\medskip
In Section \ref{sec-box}, we define the box-relation $\prec_{\sf box}$ 
on the set of LR-tableaux of the same shape. 
Proposition \ref{prop-simul-pole-decomp} 
shows how two tableaux in box relation
give rise to a simultaneous pole decomposition 
for the corresponding embeddings.

\medskip Suppose the two pole decompositions differ in two summands, 
say $R$ and $R'$ in one and $\widetilde R$ and
$\widetilde R'$ in the other embedding.  In Section \ref{sec-two-mono}, 
we define two monomorphisms, say 
$R'\to \widetilde R'$ and $\widetilde R\to R$, which 
have isomorphic cokernels (Proposition \ref{prop-monomorphisms}).
Let $Q$ be the pull-back of the cokernels of the two monomorphisms.
We show that $Q$ and $\widetilde R\oplus\widetilde R'$ have the
same LR-tableau.

\medskip
In the last Section \ref{sec-boundary} we assume that
$\Lambda$ is the power series ring with coefficients in an algebraically
closed field $k$.
Using the embedding $Q$ from the previous section,
we define a one-parameter family $Q(\mu)$, $\mu\in k$, 
of embeddings which is such that $Q(0)$, $R\oplus R'$ have the same
tableau (since the embeddings are isomorphic),
and so do $\widetilde R\oplus\widetilde R'$ and the $Q(\mu)$
for $\mu\neq 0$ (since $Q(\mu)\cong Q$ if $\mu\neq 0$).
We introduce the poset $\mathcal P_{\sf boundary}$
and use the one-parameter family to establish the two results
stated in the introduction about the boundary of the irreducible
components of invariant subspace varieties.

%======================================================================
\section{Direct sums of cyclic embeddings}
%======================================================================
\mylabel{sec-poles}

 Let $\Lambda$ be a discrete valuation ring with maximal ideal generator $p$
and radical factor field $k$.
By $\mod\Lambda$ we denote the category of all (finite length)
$\Lambda$-modules,
and by $\mathcal S(\Lambda)$ the category of all short exact sequences
in $\mod\Lambda$ with morphisms given by commutative diagrams.  
With the componentwise exact structure,
$\mathcal S(\Lambda)$ is an exact Krull-Remak-Schmidt category.
We denote objects in $\mathcal S(\Lambda)$ either as short exact sequences
$0\to A\to B\to C\to 0$ of $\Lambda$-modules, or as embeddings
$(A\subset B)$.

%---------------------------------------
\subsection{Cyclic embeddings and poles}
%---------------------------------------

\medskip
We review Kaplansky's classification of cyclic embeddings.

\medskip
An embedding $(A\subset B)$ is called {\it cyclic}
if the submodule $A$ is cyclic as a $\Lambda$-module,
that is, if $A_\Lambda$ is either indecomposable or zero.
A cyclic embedding $(A\subset B)$ 
is a {\it pole} if $A$ is an indecomposable $\Lambda$-module
and if $(A\subset B)$ is an indecomposable object in $\mathcal S(\Lambda)$.
An embedding of the form $(0\subset B)$ is called {\it empty.} 
By $E_\beta$ we denote the empty embedding
  $(0\subset N_\beta)$.

For the classification, we use height sequences which were   
originally introduced by Pr\"ufer
\cite{prue} as H\"ohenexponenten.   
In \cite[Section~18]{kap}, 
the height sequence for an~element $a\in B$ is called 
its {\it Ulm-sequence}.

\begin{defin}\begin{itemize}
  \item A {\it height sequence} is a~sequence in $\mathbb N_0\cup\{\infty\}$
      which is strictly increasing until it reaches $\infty$
      after finitely many steps.
    A~height sequence is {\it non-empty}, if it has at least
    one element from $\mathbb{N}_0$.
  \item We say an element $a\in B$, $B\in\mod\Lambda$, has {\it height} $m$
    if $a\in p^m B\setminus p^{m+1}B$.
    In this case we write $h(a)=m$.
    By definition, $h(0)=\infty$.
  \item   The {\it height sequence for $a$ in $B$} is 
    $H_B(a)=(h(p^ia))_{i\in\mathbb N_0}$.
    Sometimes, we do not write the trailing entries $\infty$.
  \end{itemize}
\end{defin}

The following result is proved in \cite[Theorem 25]{kap}.

\begin{prop}\mylabel{prop-poles}
There is a one-to-one correspondence
$$\{\text{poles}\}/_{\cong} \; \stackrel{1-1}\longleftrightarrow \;
\{\text{finite non-empty strictly increasing sequences in $\mathbb N_0$}\}.$$
\end{prop}

For the convenience of the reader we construct a~cyclic embedding $(A\subset B)$ 
for a given strictly increasing
sequence $(m_i)_{0\leq i\leq n}$ in $\mathbb N_0$. Following \cite{kap}, we say that 
$(m_i)$ has a {\it gap} after $m_\ell$ if $\ell=n$ or $m_{\ell+1}>m_\ell+1$.
Let $i_1>i_2>\cdots>i_s$ be  such that 
$(m_i)$  has gaps exactly after the entries 
$m_{i_1}>m_{i_2}>\cdots>m_{i_s}$.
For $1\leq j\leq s$ put $\beta_j=m_{i_j}+1$
and $\ell_j=m_{i_j}-i_j$, then $\beta$ and $\ell$ are strictly decreasing
sequences of positive and nonnegative integers, respectively.  
Let $$B=N_\beta=\bigoplus_{i=1}^s \Lambda/(p^{\beta_i})$$
be generated by elements $b_{\beta_j}$ of order $p^{\beta_j}$.
Let $a=\sum_{j=1}^s p^{\ell_j}\cdot b_{\beta_j}$ and put $A=(a)$. 
This yields a cyclic embedding $P((m_i))=(A\subset B)$.

\begin{ex}
  The height sequence $(1,3,4)$ has gaps after $1$ and $4$, and
  hence gives rise to the embedding
  $P((1,3,4))=((p^2b_5+pb_2)\subset N_{(5,2)})$.
  $$
  \raisebox{.5cm}{$P((1,3,4)):$} \quad
\begin{picture}(8,15)
  \multiput(0,0)(0,3)5{\smbox}
  \multiput(3,6)(0,3)2{\smbox}
  \multiput(1.5,7.5)(3,0)2\sbullet
  \put(1.5,7.5){\line(1,0)3}
\end{picture}
$$
In the picture, the columns correspond to the indecomposable direct 
summands of $B$; the $i$-th box from the top in a column of length $r$
represents the element $p^{i-1}b_r$.  The columns are aligned vertically
to make the submodule generator(s) homogeneous, if possible.
\end{ex}

  \begin{rem}
    Each indecomposable cyclic embedding is either a pole of the form
    $P((m_i))$ for $(m_i)$ a finite non-empty strictly increasing
    sequence in $\mathbb N_0$ or else an empty embedding
    of the form $E_{(n)}$ for $n$ a natural number.
  \end{rem}

\bigskip

We define {\it an~extended pole} as an~embedding that is the direct sum of a~pole and a~certain 
indecomposable empty embedding.
Extended poles arise when 
due to a ``box move'' some subsequent entries in the height sequence
are increased or decreased in such a way that a gap
disappears.

\begin{lem}\mylabel{lemma-non-gap}
Suppose that for a given height sequence $(m_i)$, the gaps are listed
in the subsequence $(m_{i_j})$,
$j=1,\ldots,s$, as above, 
but there is one non-gap included in this subsequence, 
say at $m_{i_u}$, for some $u\in\{1,\ldots,s\}$.
Let $\delta_j=m_{i_j}+1$, $k_j=m_{i_j}-i_j$, $D=N_\delta$ with
generators $d_{\delta_j}$ of order $p^{\delta_j}$  
for $j=1,\ldots,s$, 
and $c=\sum_{j=1}^s p^{k_j}d_{\delta_j}$, as above,
and denote this embedding as $P((m_i)\vee m_{i_u})=((c)\subset D)$. 

\smallskip
Then $P((m_i)\vee m_{i_u})\cong P((m_i))\oplus
E_{(\delta')}$
where $\delta'=\delta_{i_u}=m_{i_u}+1$.
\end{lem}

\begin{proof}
  Write $P((m_i))= ((a)\subset B)$ for the pole constructed
    as below Proposition~\ref{prop-poles}.
  The isomorphism in Lemma~\ref{lemma-non-gap}
  is given explicitly by the map $D\to B\oplus N_{(\delta')}$:
  $$d_{\delta_j}\mapsto \left\{\begin{array}{ll}
  b_{\delta_j}&\text{if $j\neq i_u,i_v$}\\
  d'&\text{if $j=i_u$}\\
  b_{\delta_j}-d'&\text{if $j=i_v$} \end{array}\right.$$
  if $i_v$ is the gap directly following $i_u$ (so $i_v<i_u$)
  and $N_{(\delta')}$ is generated by $d'$.
\end{proof}

\begin{ex}
  The height sequence $(1,3,4)$ has no gap after $3$, so we may consider
  the extended pole $P((1,3,4)\vee 3) = ((p^2d_5+ p^2d_4+pd_2)\subset N_{(5,4,2)})$:
  $$
  \raisebox{.5cm}{$P((1,3,4)\vee 3):$} \quad
\begin{picture}(11,15)
  \multiput(0,0)(0,3)5{\smbox}
  \multiput(3,3)(0,3)4{\smbox}
  \multiput(6,6)(0,3)2{\smbox}
  \multiput(1.5,7.5)(3,0)3\sbullet
  \put(1.5,7.5){\line(1,0)6}
\end{picture}
\qquad
  \raisebox{.5cm}{$\cong P((1,3,4))\oplus (0\subset N_{(4)}):$} \quad
\begin{picture}(8,15)
  \multiput(0,0)(0,3)5{\smbox}
  \multiput(3,6)(0,3)2{\smbox}
  \multiput(1.5,7.5)(3,0)2\sbullet
  \put(1.5,7.5){\line(1,0)3}
\end{picture}
\raisebox{.5cm}{$\oplus$}\;
\begin{picture}(3,15)
  \multiput(0,1.5)(0,3)4{\smbox}
\end{picture}
$$
\end{ex}

%----------------------------------
\subsection{Partitions, tableaux and partial maps} 
%----------------------------------

\mylabel{sec-tableaux}

We introduce the combinatorial tools needed in this paper:
partitions, tableaux, and partial maps.

\medskip
An embedding $(A\subset B)$ with corresponding short exact
sequence $0\to A\to B\to C\to 0$ gives rise to three partitions
$\alpha,\beta,\gamma$ which describe the isomorphism types of 
$A,B,C$, respectively.
Recall that the partition $\alpha=(\alpha_1,\ldots,\alpha_r)$
(where $\alpha_1\geq\cdots\geq\alpha_r\geq1$) corresponds to the 
$\Lambda$-module $N_\alpha=\bigoplus_{i=1}^r\Lambda/(p^{\alpha_i})$.

\medskip
The Green-Klein Theorem \cite{klein1} states that a partition triple
$(\alpha,\beta,\gamma)$ can be realized in this way by a short exact 
sequence if and only if there exists an LR-tableau
of shape $\beta\setminus\gamma$ and content $\alpha$.
By definition, an {\it LR-tableau} is a Young diagram of shape $\beta$
(the {\it outer shape} of the tableau) in which 
the region $\beta\setminus\gamma$ is filled with $(\alpha')_1$ entries
1, $(\alpha')_2$ entries 2, etc., where $\alpha'$ denotes the transpose 
of $\alpha$, such that three conditions are satisfied:
In each row, the entries are weakly increasing; in each column, the entries
are strictly increasing; and the lattice permutation property holds, that is,
on the right hand side of each column there occur at least as many entries
$e$ as there are entries $e+1$ ($e=1,2,\ldots$).

\medskip
The theorem is constructive in the sense that for each LR-tableau $\Gamma$
there is an embedding $(A\subset B)$ with corresponding tableau $\Gamma$.
The tableau of an embedding is obtained as follows.

\medskip

Suppose that the submodule $A$ has Loewy
length $r$.  Then the sequence of epimorphisms
$$B=B/p^rA\twoheadrightarrow B/p^{r-1}A\twoheadrightarrow \cdots\twoheadrightarrow B/pA
\twoheadrightarrow B/A$$
gives rise to a sequence of inclusions of partitions
$$\beta=\gamma^{(r)}\supset\gamma^{(r-1)}\supset\cdots
  \supset\gamma^{(1)}\supset\gamma^{(0)}=\gamma,$$
where the $e$-th partition records the lengths of the 
indecomposable summands of $B/p^eA$.
The partitions define the corresponding LR-tableau
$\Gamma=[\gamma^{(0)}, \gamma^{(1)},\ldots,\gamma^{(r)}]$ as follows.
The Young diagram of shape $\beta=\gamma^{(r)}$ is filled in such a way
that the region given by $\gamma=\gamma^{(0)}$ remains empty and 
every box in the skew diagram
$\gamma^{(e)}\setminus\gamma^{(e-1)}$ carries the entry $\singlebox e$.
The {\it shape} of the tableau is the skew diagram $\beta\setminus\gamma$,
the {\it content} is the partition $\alpha$.

\medskip
The {\it union} of two tableaux is taken row-wise, so if
$E=[\varepsilon^{(i)}]_{0\leq i\leq s}$ and $F=[\zeta^{(i)}]_{0\leq i\leq t}$ 
are tableaux then
the partitions $\gamma^{(i)}$ for $\Gamma=E\cup F$ are given by taking
the union (of ordered multi-sets):  
$\gamma^{(i)}=\varepsilon^{(i)}\cup\zeta^{(i)}$ 
where $\varepsilon^{(i)}=\varepsilon^{(s)}$
for $i\geq s$ and $\zeta^{(i)}=\zeta^{(t)}$ for $i\geq t$.

\smallskip
It is easy to see that if embeddings $(A\subset B)$, $(C\subset D)$
have tableaux $E$, $F$, respectively, then the direct sum
$(A\subset B)\oplus (C\subset D)$ has tableau 
$\Gamma=E\cup F$.

\medskip
We can also decompose some tableaux into column 
tableaux:

\smallskip
{\it Column tableaux} (which we call {\it columns} if the context is clear) 
are tableaux which have only one column, and where the entries
in this column are subsequent natural numbers (not necessarily starting at 1).
Each column is one of the following: 
For $1\leq e\leq f$, we denote by $C(e,f)_n$ 
the 1-column tableau of height $n$ with entries $e,\ldots,f$.
Formally, $C(e,f)_n=[\gamma^{(0)},\ldots,\gamma^{(f)}]$ where
$\gamma^{(0)}=\cdots=\gamma^{(e-1)}=(n-f+e-1), \gamma^{(e)}=(n-f+e),\ldots,\gamma^{(f)}=(n)$.  
By $C(1,0)_n$ we denote the empty column of length $n$.
Note that a column tableau $C(e,f)_n$ satisfies
the lattice permutation property, and hence is an LR-tableau,
  if and only if $e=1$.

\begin{ex}
In a union of column tableaux, the vertical columns
need no longer form  column tableaux.
  $$
  \begin{picture}(3,9)
    \multiput(0,0)(0,3)3{\smbox}
    \put(0,6){\numbox1}
    \put(0,3){\numbox2}
    \put(0,0){\numbox3}
  \end{picture} \quad
  \raisebox{.35cm}{$\cup$} \quad
  \begin{picture}(3,9)
    \multiput(0,3)(0,3)2{\smbox}
    \put(0,3){\numbox1}
  \end{picture}
  \quad\raisebox{.35cm}{$=$} \quad
  \begin{picture}(6,9)
    \multiput(0,0)(0,3)3{\smbox}
    \multiput(3,3)(0,3)2{\smbox}
    \put(3,6){\numbox1}
    \put(3,3){\numbox2}
    \put(0,3){\numbox1}
    \put(0,0){\numbox3}
  \end{picture}
  $$
\end{ex}

\begin{defin}
  A {\it partial map} $g$ on an LR-tableau $\Gamma$ assigns to each box
  $\singlebox e$ with $e>1$ a box $\singlebox d$ with entry $d=e-1$ such that
  \begin{enumerate}
  \item $g$ is one-to-one,
  \item for each box $b$, the row of $g(b)$ is above the row of $b$. 
  \end{enumerate}
    An {\it orbit} of a partial map $g$ is a sequence 
    $\singlebox e, g(\singlebox e), g^2(\singlebox e),\ldots
    , g^{e-1}(\singlebox e)$ given by a box $\singlebox e$ 
      which is not in the image of $g$.
      Thus, the orbits of $g$ are in one-to-one correspondence with
      the boxes $\singlebox e$ not in the image of $g$ (the correspondence
      is given by taking the first box in the orbit),
      and in one-to-one correspondence with the boxes that have entry 1
      (where the correspondence is given by taking the last box).
\end{defin}

\begin{rem}
  Given an LR-tableau $\Gamma$, the existence of at least one partial map
  follows from the lattice permutation property.
\end{rem}

Given a partial map $g$ on an LR-tableau $\Gamma$, a {\it jump} in row $r$ 
is a box $b$ in this row with the property that either $b$ 
has entry $\singlebox 1$ or that $g(b)$ occurs in row $s$ with $s<r-1$. 
We say a partial map $g$ on $\Gamma$ has the {\it empty box property (EBP)}
if for each $r$, there are at least as many columns in 
$\Gamma$ of exactly $r-1$ empty boxes as there are jumps in row $r$.

\begin{ex}
Consider the following LR-tableau
$$
 \raisebox{.5cm}{$\Gamma:$} \quad
  \begin{picture}(12,12)
    \multiput(0,0)(0,3)4{\smbox}
    \multiput(3,3)(0,3)3{\smbox}
    \multiput(6,6)(0,3)2{\smbox}
    \put(9,9){\smbox}
    \put(9,9){\numbox1}
    \put(6,9){\numbox1}
    \put(6,6){\numbox2}
    \put(0,3){\numbox1}
    \put(3,3){\numbox2}
    \put(0,0){\numbox3}
  \end{picture}
$$  
To define a~partial map, we have to assign to each box 
$\singlebox 3$ a corresponding box $\singlebox 2$ and  to each box $\singlebox 2$ 
a corresponding box $\singlebox 1$. 
Note that we can do this in four different ways. 
To specify the maps, we distinguish boxes 
of the same entry by indicating the row on the right;
for the two boxes in the first row, we use the letters
 L and R.
Consider the partial map $g$ defined as follows:
$$g:\boxwithrow34\mapsto\boxwithrow23,\; \boxwithrow23\mapsto \boxwithrow1L,\;
\boxwithrow22\mapsto \boxwithrow1R$$
Note that the partial map $g$ has three orbits
(one consists only of the box $\boxwithrow 13$).
It satisfies (EBP) since there is a column of 2 empty boxes
corresponding to the jump $\boxwithrow23\mapsto \boxwithrow 11$.
  However, the partial map
$$g': \boxwithrow34\mapsto \boxwithrow22,\; \boxwithrow23\mapsto \boxwithrow1L,\;
\boxwithrow22\mapsto \boxwithrow1R$$
does not satisfy (EBP), because there is no column of 3 empty boxes
corresponding to the jump at $\boxwithrow 34\mapsto \boxwithrow22$.

\end{ex}
 
  \begin{lem}
    \mylabel{lemma-tableau-of-cyclic}
    Suppose the embedding $(A\subset B)$ is cyclic with $A$ a module
    of Loewy length $r$.
    The height sequence $(m_i)$
    of the submodule generator $a$ 
    determines the rows in $\Gamma$ in which
    the entries occur, and conversely.  
    More precisely,
      for $1\leq e\leq r$,\
    the entry $e$ occurs
      once, namely
    in row $m_{e-1}+1$.
      Moreover, a~tableau with this property is a union of column tableaux.
  \end{lem}

  \begin{proof}
    For any embedding, the number of boxes in the first $m$ rows of
    $\gamma^{(e)}$
    is the length of $(B/p^eA)/((p^mB+p^eA)/p^eA)=B/(p^eA+p^mB)$.
    Here, the modules $B/(p^{e-1}A+p^mB)$ and $B/(p^eA+p^mB)$
    have the same length
    if $m\leq m_{e-1}$ and lengths differing by one otherwise
    since $p^{e-1}a\in p^{m_{e-1}}B\setminus p^{m_{e-1}+1}B$.
    The last statement it easy to deduce.
  \end{proof}

%---------------------------------
 \subsection{Direct sums of cyclic embeddings}
%---------------------------------

\mylabel{sec-sums-cyclic}

  The classification of cyclic embeddings by Kaplansky
  can be generalized.
  An algebraic description of the multiplicity of a certain pole
  as a direct summand of a direct sum of cyclic embeddings
  is given in \cite[in particular Lemma~1 and Theorem~2]{hrw}.
  In this section, we present a combinatorial description
  in terms of partial maps on Littlewood-Richardson tableaux.
  We will revisit this topic in the upcoming paper \cite{ks-cyclic}.

\begin{defin}\begin{itemize}  
  \item
    Two partial maps $g$, $g'$ on an LR-tableau $\Gamma$ are {\it equivalent}
    if $g'=\pi^{-1}\circ g\circ \pi$ for some permutation $\pi$ 
    on the set of non-empty  boxes  in $\Gamma$
    which preserves the entry and the row.
  \item Two pairs $(\Gamma,g)$, $(\Delta, h)$, each consisting of an
    LR-tableau and a partial map on the tableau, are {\it equivalent}
    if $\Gamma=\Delta$ and if the partial maps $g,h$ are equivalent.
    In this case we write $(\Gamma,g)\sim(\Delta,h)$.
  \end{itemize}
\end{defin}

 \begin{ex}
  Consider the LR-tableau from the example in Section \ref{sec-tableaux}.
$$
 \raisebox{.5cm}{$\Gamma:$} \quad
  \begin{picture}(12,12)
    \multiput(0,0)(0,3)4{\smbox}
    \multiput(3,3)(0,3)3{\smbox}
    \multiput(6,6)(0,3)2{\smbox}
    \put(9,9){\smbox}
    \put(9,9){\numbox1}
    \put(6,9){\numbox1}
    \put(6,6){\numbox2}
    \put(0,3){\numbox1}
    \put(3,3){\numbox2}
    \put(0,0){\numbox3}
  \end{picture}
$$  
The maps $g$ and $g'$ are not equivalent.  
But $g$ is equivalent to the partial map
$$h: \singlebox3\mapsto\boxwithrow23, \;\boxwithrow23\mapsto \boxwithrow1R,\;
\boxwithrow22\mapsto \boxwithrow1L.$$
(There is a fourth partial map on $\Gamma$, it is equivalent to $g'$.)
 \end{ex}

\medskip
Clearly, given two equivalent partial maps on an LR-tableau,
then one satisfies (EBP) if and only if the other does.\medskip

Here is our generalization of Kaplansky's result.

\begin{thm}\mylabel{theorem-combinatorial-classification}
  There is a one-to-one correspondence
    $$\big\{\bigoplus\text{cyclics}\big\} \big /_{\cong}
    \quad \stackrel{1-1}\longleftrightarrow \quad
    \big\{(\Gamma,g):\text{$g$ partial map on $\Gamma$ with (EBP)}\big\}
    \big /_{\sim}$$
  between the isomorphism types
  of direct sums of cyclic embeddings, and the equivalence classes
  of pairs $(\Gamma,g)$ where $\Gamma$ is an LR-tableau
  and $g$ a partial map on $\Gamma$ which satisfies (EBP).
\end{thm}

  \begin{proof}
    We describe this correspondence explicitely.
    
    \smallskip
    Consider a direct sum of cyclic embeddings as the sum
    of an empty embedding and a direct sum of poles.
    The height sequence of each pole gives rise to a partial map
    on the tableau of the pole which satisfies (EBP),
    this map has precisely one orbit; it records the isomorphism
    type of the pole (Lemma~\ref{lemma-tableau-of-cyclic}).
    The tableau of the direct sum is given by taking the row-wise union
    of the tableaux of the summands; it admits a partial map which is 
    (boxwise) defined by the partial maps for the summands; this map satisfies
    (EBP) because the restricted maps do. 
    Given two isomorphic direct sum decompositions, the two unordered lists of
    height sequences of the poles involved are equal, hence the 
    two associated partial maps differ by conjugation by a
    permutation of the boxes
    which preserves rows and entries.

    \smallskip
    Conversely, suppose $g$ is a partial map on $\Gamma$ with (EBP).
    For each orbit $\mathcal O$ of $g$, the boxes in $\mathcal O$ together
    with the empty parts of columns which correspond to the jumps 
    (recall the definition of the (EBP)), 
    constitute the tableau of a pole.  
    The empty parts of columns in $\Gamma$
    not used by the (EBP) define the tableau of an empty embedding. 
    Since all boxes are accounted for, $\Gamma$ is the union of all the 
    tableaux.  Hence the sum of the poles and the empty embedding 
    has $\Gamma$ as its tableau.
    If $g'$ is an equivalent partial map on $\Gamma$, then the boxes in the 
    orbits may differ, but not the rows in which they occur.
    Thus the corresponding
    unordered list of height sequences of poles is equal, 
    and so is the partition
    which records those empty parts of columns which are not used up
    by the jumps in the height sequences.  
    Hence the associated embeddings are isomorphic.
    
    \smallskip
    The two assignments are inverse to each other since this is true for 
    poles and empty embeddings.
  \end{proof}

There is a different way to read off from a given LR-tableau
if it is the tableau of a direct sum of cyclic embeddings.

\begin{lem}
\mylabel{lemma-union-columns}
An LR-tableau $\Gamma$ is the tableau of 
a direct sum of cyclic embeddings if and only if $\Gamma$ is
a union of columns.
\end{lem}

  \begin{proof}
    In Lemma~\ref{lemma-tableau-of-cyclic} 
    we have seen that  the tableau of a cyclic embedding  is 
    a union of columns, and this property is preserved under taking
    direct sums.
    
    \smallskip
    Conversely, if the tableau is a union of columns, then the 
    lattice permutation property allows for the construction of a 
    partial map with (EBP).  The result follows from 
    Theorem~\ref{theorem-combinatorial-classification}.
  \end{proof}

We illustrate presented results and notation with examples.

\begin{ex}
In general, the embedding corresponding to a given LR-tableau
is not determined uniquely, up to isomorphy, by the tableau alone.
In each case, the following embeddings will have the same given LR-tableau.

\medskip
(1) Two nonisomorphic sums of poles, given by nonequivalent partial maps.
  $$
  \raisebox{.5cm}{$\Gamma:$} \;
  \begin{picture}(8,12)
    \multiput(0,0)(0,3)4{\smbox}
    \multiput(3,6)(0,3)2{\smbox}
    \put(6,9){\numbox1}
    \put(3,6){\numbox1}
    \put(0,0){\numbox2}
  \end{picture}
  \quad
  \raisebox{.5cm}{$P((1,3))\oplus P((0)):$} \;
  \begin{picture}(12,12)
    \multiput(0,0)(0,3)4{\smbox}
    \multiput(3,3)(0,3)2{\smbox}
    \multiput(1.5,4.5)(3,0)2\sbullet
    \put(1.5,4.5){\line(1,0)3}
    \put(9,3){\smbox}
    \put(10.5,4.5)\sbullet
  \end{picture}
  \quad
  \raisebox{.5cm}{$P((0,3))\oplus P((1)):$} \;
  \begin{picture}(12,12)
    \multiput(0,0)(0,3)4{\smbox}
    \multiput(3,3)(0,3)1{\smbox}
    \multiput(1.5,4.5)(3,0)2\sbullet
    \put(1.5,4.5){\line(1,0)3}
    \multiput(9,3)(0,3)2{\smbox}
    \put(10.5,4.5)\sbullet
  \end{picture}
  $$

(2) A sum of two poles vs.\ a sum of two poles and an empty embedding.
  $$
  \raisebox{.5cm}{$\Gamma:$} \;
  \begin{picture}(8,9)
    \multiput(0,0)(0,3)3{\smbox}
    \multiput(3,6)(0,3)1{\smbox}
    \put(6,6){\numbox1}
    \put(3,3){\numbox1}
    \put(0,0){\numbox2}
  \end{picture}
  \quad
  \raisebox{.5cm}{$P((0,2))\oplus P((1)):$} \;
  \begin{picture}(12,9)
    \multiput(0,0)(0,3)3{\smbox}
    \multiput(3,3)(0,3)1{\smbox}
    \multiput(1.5,4.5)(3,0)2\sbullet
    \put(1.5,4.5){\line(1,0)3}
    \multiput(9,3)(0,3)2{\smbox}
    \put(10.5,4.5)\sbullet
  \end{picture}
  \quad
  \raisebox{.35cm}{\parbox{.23\textwidth}{%
      $P((1,2))\oplus P((0))$\\ $\phantom x\quad\oplus E_{(2)}:$}} 
  \;
  \begin{picture}(15,9)
    \multiput(0,0)(0,3)3{\smbox}
    \multiput(6,3)(0,3)1{\smbox}
    \multiput(1.5,4.5)(3,0)1\sbullet
     \multiput(12,3)(0,3)2{\smbox}
    \put(7.5,4.5)\sbullet
  \end{picture}
  $$

(3) The sum of poles is determined uniquely by the partial map; 
but there is also an embedding which is not a sum of poles.
  $$
  \raisebox{.5cm}{$\Gamma:$} \;
  \begin{picture}(8,12)
    \multiput(0,0)(0,3)4{\smbox}
    \multiput(3,3)(0,3)3{\smbox}
    \multiput(6,6)(0,3)2{\smbox}
    \put(6,9){\numbox1}
    \put(6,6){\numbox2}
    \put(3,3){\numbox1}
    \put(0,0){\numbox3}
  \end{picture}
  \quad
  \raisebox{.5cm}{$P((0,1,3))\oplus P((2)):$} \;
  \begin{picture}(12,12)
    \multiput(0,0)(0,3)4{\smbox}
    \multiput(3,3)(0,3)2{\smbox}
    \multiput(1.5,7.5)(3,0)2\sbullet
    \put(1.5,7.5){\line(1,0)3}
    \multiput(9,3)(0,3)3{\smbox}
    \put(10.5,4.5)\sbullet
  \end{picture}
  \quad
  \raisebox{.5cm}{$E\oplus E_{(3)}:$} \;
  \begin{picture}(12,12)
    \multiput(0,0)(0,3)4{\smbox}
    \multiput(3,3)(0,3)2{\smbox}
    \multiput(1.5,7.5)(3,0)2\sbullet
    \put(4.5,4.5)\sbullet
    \put(1.5,7.5){\line(1,0)3}
    \multiput(9,3)(0,3)3{\smbox}
  \end{picture}
  $$

(4) Two partial maps, up to equivalence, only one has the (EBP).
  $$
  \raisebox{.5cm}{$\Gamma:$} \;
  \begin{picture}(12,12)
    \multiput(0,0)(0,3)4{\smbox}
    \multiput(3,3)(0,3)3{\smbox}
    \multiput(6,6)(0,3)2{\smbox}
    \put(9,9){\smbox}
    \put(9,9){\numbox1}
    \put(6,9){\numbox1}
    \put(6,6){\numbox2}
    \put(0,3){\numbox1}
    \put(3,3){\numbox2}
    \put(0,0){\numbox3}
  \end{picture}
  \quad
  \raisebox{.5cm}{\parbox{.23\textwidth}{%
      $P((0,1,3))$\\ $\phantom x\quad \oplus P((2))$\\
      $\phantom x\quad \oplus P((0,1)):$}}
  \begin{picture}(18,12)
    \multiput(0,0)(0,3)4{\smbox}
    \multiput(3,3)(0,3)2{\smbox}
    \multiput(1.5,7.5)(3,0)2\sbullet
    \put(1.5,7.5){\line(1,0)3}
    \multiput(9,3)(0,3)3{\smbox}
    \put(10.5,4.5)\sbullet
    \multiput(15,3)(0,3)2{\smbox}
    \put(16.5,7.5)\sbullet
  \end{picture}
  \quad
  \raisebox{.5cm}{$E\oplus P((0,2)):$} \;
  \begin{picture}(15,12)
    \multiput(0,0)(0,3)4{\smbox}
    \multiput(3,3)(0,3)2{\smbox}
    \multiput(1.5,7.5)(3,0)2\sbullet
    \put(4.5,4.5)\sbullet
    \put(1.5,7.5){\line(1,0)3}
    \multiput(9,3)(0,3)3{\smbox}
    \put(12,6){\smbox}
    \multiput(10.5,7.5)(3,0)2\sbullet
    \put(10.5,7.5){\line(1,0)3}
  \end{picture}
  $$

(5) Two partial maps with (EBP), up to equivalence.
  $$
  \raisebox{.5cm}{$\Gamma:$} \;
  \begin{picture}(12,12)
    \multiput(0,0)(0,3)4{\smbox}
    \multiput(3,3)(0,3)3{\smbox}
    \multiput(6,3)(0,3)3{\smbox}
    \put(9,9){\smbox}
    \put(9,9){\numbox1}
    \put(6,6){\numbox1}
    \put(3,3){\numbox2}
    \put(6,3){\numbox2}
    \put(0,0){\numbox3}
  \end{picture}
  \quad
 \raisebox{.5cm}{\parbox{.23\textwidth}{%
      $P((0,2,3))$\\ $\phantom x\quad \oplus P((1,2))$\\
      $\phantom x\quad \oplus E_{(3)}:\;$}}
  \begin{picture}(18,12)
    \multiput(0,0)(0,3)4{\smbox}
    \multiput(3,6)(0,3)1{\smbox}
    \multiput(1.5,7.5)(3,0)2\sbullet
    \put(1.5,7.5){\line(1,0)3}
    \multiput(9,3)(0,3)3{\smbox}
    \put(10.5,7.5)\sbullet
    \multiput(15,3)(0,3)3{\smbox}
  \end{picture}
  \quad
 \raisebox{.5cm}{\parbox{.23\textwidth}{%
      $P((1,2,3))$\\ $\phantom x\quad \oplus P((0,2))$\\
      $\phantom x\quad \oplus E_{(3)}:\;$}}
  \begin{picture}(12,12)
    \multiput(0,0)(0,3)4{\smbox}
    \put(1.5,7.5){\sbullet}
    \multiput(6,3)(0,3)3{\smbox}
    \put(9,6){\smbox}
    \multiput(7.5,7.5)(3,0)2\sbullet
    \put(7.5,7.5){\line(1,0)3}
    \multiput(15,3)(0,3)3{\smbox}
  \end{picture}
  $$

\end{ex}

\begin{defin}
An element $a\in B$ defines a filtration for $B$, as follows.
Let $\mathcal F_0=\End(B)\cdot a$ and, for $i\in\mathbb Z\setminus\{0\}$, 
put $\mathcal F_i=p^{-i}\mathcal F_0$.
We call $(\mathcal F_i)_i$ the {\it filtration of $B$ centered at $a$}.
\end{defin}

Clearly, the filtration 
$\cdots\subset \mathcal F_{-1}\subset 
\mathcal F_0\subset\mathcal F_{1}\subset\cdots$ 
has factors which are
semisimple $\Lambda$-modules, and $a\in\mathcal F_i$ for $i\geq 0$.

\begin{lem}\mylabel{lemma-filtration}
Suppose $a=\sum p^{\ell_j} b_{\beta_j}$ is the submodule generator of a pole
or extended pole.
\begin{enumerate}
\item The elements $p^{\ell_j} b_{\beta_j}$ form a minimal generating set 
  for the $\Lambda$-module $\mathcal F_0$.
\item For $\alpha\in\mathbb Z$, 
  the $k$-vector space $\mathcal F_\alpha/\mathcal F_{\alpha-1}$
  has basis given by the residue classes of the elements 
  $p^{\ell_j-\alpha} b_{\beta_j}$ where $j$ is such that 
  $\ell_j-\beta_j< \alpha\leq \ell_j$.
\item The residue class of the submodule generator $a$ is
  homogeneous in $\mathcal F_0/\mathcal F_{-1}$ in the sense that it is 
  the sum of the basic elements in 2.
\item Similarly, for $\alpha>0$, the residue class of $p^\alpha a$ 
  is a homogeneous element in $\mathcal F_{-\alpha}/\mathcal F_{-\alpha-1}$.
\end{enumerate}
\end{lem}

\begin{proof}
1. For a pole $((a)\subset B)$ with $a=\sum p^{\ell_j} b_{\beta_j}$,
the sequences $\ell_j$ and $k_j=\beta_j-\ell_j$ are strictly decreasing,
hence a term $p^{\ell_u} b_{\beta_u}$ occurs in exactly one summand 
in $\End(B)\cdot a=\sum(\rad^{\ell_j}B\cap\soc^{k_j}B)$.
If the pole is extended, then the sequence $k_j$ is still strictly 
decreasing, but $\ell_j$ is stationary at one place, say $\ell_v=\ell_{v-1}$.
The corresponding summand $p^{\ell_v} b_{\beta_v}$ occurs in exactly two
terms in the expression $\End(B)\cdot a=\sum(\rad^{\ell_j}B\cap\soc^{k_j}B)$,
in each case, it is not in the radical.

2. The second statement follows from the first.

3.\ and 4.  The last two statements follow from the first two.
\end{proof}

%==============================================================================
\section{Box moves}
%==============================================================================
\mylabel{sec-box}

We are interested in the ``transition'' between 
the collections of embeddings which are given by two different but ``similar'' LR-tableaux.
In this section we show that whenever the two tableaux differ by a so called box move, 
then they can be realized by embeddings which have ``compatible'' decompositions
as direct sums of cyclic embeddings.

\begin{defin}\mylabel{def-box-move}
Suppose that $\Gamma$, $\widetilde\Gamma$ are LR-tableaux 
of the same shape and content, and
that both are unions 
of column tableaux
in such a way that they differ in exactly two of those columns.
 We say $\widetilde\Gamma$ is obtained from $\Gamma$ by an {\it increasing  box move}
 if these two columns where they differ are of the form
 $C=C(e,f)_n$, $C'=C(e',f')_{n'}$ in $\Gamma$ 
 and $\widetilde C=C(e',f')_n$, $\widetilde C'=C(e,f)_{n'}$
 in $\widetilde\Gamma$, for suitable  
numbers $n>n'$, $e<e'$, $f<f'$ such that $f-e=f'-e'$.
  That is, the columns $C$ and $C'$ contain the same number
  of subsequent entries,
  but column $C$ is longer and contains the smaller entries.
  After the increasing box move, column $\widetilde C$ has the
  same length as $C$ but the entries are taken from $C'$;
  similarly, $\widetilde C'$ has the length of $C'$ but the entries from $C$.
  Thus, in an increasing box move, the smaller entries move up.
  By reversing an increasing box move we obtain
  a {\it decreasing box move;} here, the smallest entry of the two
  affected columns moves towards a lower position.
We write $\Gamma\leq_{\sf box}\widetilde\Gamma$ where $\leq_{\sf box}$ denotes the reflexive and
transitive closure of the relation given by increasing box moves.
\end{defin}

\begin{ex}
1. The name {\it box move} originates from dealing with LR-tableaux which are
horizontal and vertical strips; the 
box move is simply given by exchanging two boxes in the tableau.  
In the example in the introduction, $\Gamma_2$ is obtained from $\Gamma_1$ by an increasing box move.
$$\raisebox{4mm}{$\Gamma_1:$} \quad
\begin{picture}(15,12)(0,3)
\multiput(0,12)(3,0)4{\smbox}
\put(12,12){\numbox{1}}
\multiput(0,9)(3,0)3{\smbox}
\put(9,9){\numbox{2}}
\multiput(0,6)(3,0)2{\smbox}
\put(6,6){\numbox{3}}
\put(0,3){\smbox}
\put(3,3){\numbox{1}}
\put(0,0){\numbox{2}}
\end{picture}
\quad\raisebox{4mm}{$<_{\sf box}$} \;\;\;
\raisebox{4mm}{$\Gamma_{2}:$} \quad
\begin{picture}(15,12)(0,3)
\multiput(0,12)(3,0)4{\smbox}
\put(12,12){\numbox{1}}
\multiput(0,9)(3,0)3{\smbox}
\put(9,9){\numbox{2}}
\multiput(0,6)(3,0)2{\smbox}
\put(6,6){\numbox{1}}
\put(0,3){\smbox}
\put(3,3){\numbox{3}}
\put(0,0){\numbox{2}}
\end{picture}
$$
(We exchanged entries $1$ and $3$ in columns $2$ and $3$.)

2. Both tableaux $\Gamma=C(1,2)_5\cup C(3,4)_4\cup C(1,2)_2$ and
$\widetilde\Gamma=C(3,4)_5\cup C(1,2)_4\cup C(1,2)_2$ are unions 
of columns (the first in two different ways).
$\widetilde\Gamma$ is obtained from $\Gamma$ by an increasing box move:  
$\Gamma<_{\sf box}\widetilde\Gamma$.
  $$
  \raisebox{.5cm}{$\Gamma:$} \quad
  \begin{picture}(9,15)
    \multiput(0,0)(0,3)5{\smbox}
    \multiput(3,3)(0,3)4{\smbox}
    \multiput(6,9)(0,3)2{\smbox}
    \put(6,12){\numbox1}
    \put(6,9){\numbox2}
    \put(3,6){\numbox3}
    \put(3,3){\numbox4}
    \put(0,3){\numbox1}
    \put(0,0){\numbox2}
  \end{picture}
  \quad
  \raisebox{.5cm}{$\widetilde\Gamma:$} \quad
  \begin{picture}(9,15)
    \multiput(0,0)(0,3)5{\smbox}
    \multiput(3,3)(0,3)4{\smbox}
    \multiput(6,9)(0,3)2{\smbox}
    \put(6,12){\numbox1}
    \put(6,9){\numbox2}
    \put(3,6){\numbox1}
    \put(3,3){\numbox3}
    \put(0,3){\numbox2}
    \put(0,0){\numbox4}
  \end{picture}
  $$
\end{ex}

Our aim is to show:

\begin{prop}\mylabel{prop-simul-pole-decomp}
Suppose $\Gamma$, $\widetilde\Gamma$ are LR-tableaux 
of the same shape and content such that
$\widetilde\Gamma$ is obtained from $\Gamma$ by an increasing box move.
Then each of the tableaux $\Gamma$, $\widetilde\Gamma$ is 
the tableau of a sum of poles and an empty embedding.
These direct sum decompositions differ, up to isomorphy and reordering of
the summands, in exactly two poles and perhaps in the empty embeddings.
\end{prop}

\begin{proof}
  Suppose $\widetilde\Gamma$ is obtained from $\Gamma$ by an
    increasing box move.  By definition, the two tableaux are unions of column
    tableaux which differ in columns of type 
    $C(e,f)_n$, $C(e',f')_{n'}$ for $\Gamma$ and $C(e',f')_n$, $C(e,f)_{n'}$ for
    $\widetilde\Gamma$, where $n>n'$ and $e<e'$.
    We identify the complements $\Gamma\setminus(C(e,f)_n\cup C(e',f')_{n'})$
    and $\widetilde\Gamma\setminus(C(e',f')_n\cup C(e,f)_{n'})$.
    
    \smallskip
    Our key tool is Theorem~\ref{theorem-combinatorial-classification}.
    Since $\Gamma$, $\widetilde\Gamma$ are LR-tableaux for direct sums of cyclic embeddings,
    there exist partial maps, $h$ on $\Gamma$ and $\widetilde h$ on $\widetilde\Gamma$,
    which both satisfy the (EBP).
    We define new partial maps, $g$ on $\Gamma$ and $\widetilde g$ on $\widetilde\Gamma$
    which also satisfy the (EBP) and which, in addition, differ in exactly two orbits.

    \smallskip
    The two orbits where $g$ and $\widetilde g$ differ are as follows.
    The orbits of $g$ which contain the non-empty boxes in $C(e,f)_n$ and in
    $C(e',f')_{n'}$, correspond 
    to the orbits of $\widetilde g$ which contain the non-empty boxes 
    in $C(e,f)_{n'}$
    and in $C(e',f')_n$, respectively.
    
    \smallskip
    We define the new partial maps entry by entry on the boxes in the columns.
    
    \smallskip
    For a partial map $g$ and $x\geq 2$, denote by $g_x$ the restriction of the domain of $g$
    to boxes of the form $\singlebox x$. Clearly, the map $g$ is given by the sequence of its 
    restrictions $(g_x)_{x\geq 2}$.
    Depending on the subscript $x$, we use either the map $h_x$ 
    to define first  $g_x$, and then $\widetilde g_x$;
    or else we use the map ${\widetilde h}_x$
    to define first $\widetilde g_x$, and then $g_x$.

    \smallskip
    For an entry $x$ different from $e$, $e'$, $f+1$, $f'+1$,
    put $g_x=h_x$.
    For boxes $b$ with entry $x$, put
    $$\widetilde g_x(b)=\left\{\begin{array}{ll}
    \text{the box above $b$} & \text{if}\; b\in C(e',f')_n\cup C(e,f)_{n'}\\
    g_x(b) & \text{otherwise.}\\
    \end{array}\right.$$

    \smallskip
    If $e>1$, define $\widetilde g_e=\widetilde h_e$.
    Let $b'$ be the box $\singlebox e$
    in $C(e,f)_{n'}$.  For boxes $b$ with entry $e$ put
    $$g_e(b)=\left\{\begin{array}{ll}
    \widetilde g_e(b') & \text{if $b$ is the box $\singlebox e$ in $C(e,f)_n$}\\
    \widetilde g_e(b) & \text{otherwise.}
    \end{array}\right.$$
    (Since $n>n'$, the box $b$ with entry $e$ in $C(e,f)_n$
    is lower in $\Gamma$ than the box $b'$
    in $\widetilde \Gamma$.  Hence we pick $\widetilde g_e$ first
    and then define $g_e(b)=\widetilde g_e(b')$.)

    \smallskip
    Next, define $\widetilde g_{f'+1}=\widetilde h_{f'+1}$.
    There is a unique box $b'$ such that $\widetilde g_{f'+1}(b')$ is
    the box $\singlebox{f'}$ in $C(e',f')_{n}$.  For a box $b$ with entry $f'+1$ define
    $$g_{f'+1}(b)=\left\{\begin{array}{ll}
    \text{the box $\singlebox{f'}$ in $C(e',f')_{n'}$} & \text{if $b=b'$}\\
    \widetilde g_{f'+1}(b) & \text{otherwise.}
    \end{array}\right.$$
    
    Assume first that $e'\neq f+1$.  Define $g_{e'}=h_{e'}$.  Let $b'$ be the box 
    $\singlebox{e'}$ in $C(e',f')_{n'}$.  For boxes $b$ with entry $e'$ put
    $$\widetilde g_{e'}(b)=\left\{\begin{array}{ll}
    g_e(b') & \text{if $b$ is the box $\singlebox{e'}$ in $C(e',f')_n$} \\
    \text{the box above $b$} & \text{if $e<e'<f$ and $b\in C(e,f)_{n'}$}\\
    g_e(b) & \text{otherwise.}\\
    \end{array}\right.$$
    
    Define $g_{f+1}=h_{f+1}$.  There is a unique box $b'$ such that $g_{f+1}(b')$ is 
    the box $\singlebox f$ in $C(e,f)_n$.
    For boxes $b$ with entry $f+1$ define
    $$\widetilde g_{f+1}(b)=\left\{\begin{array}{ll} 
    \text{the box $\singlebox{f}$ in $C(e,f)_{n'}$} & \text{if}\; b=b'\\
    \text{the box above $b$} & \text{if}\; b\in C(e',f')\setminus\singlebox{e'}\\
    g_{f+1}(b) & \text{otherwise}.
    \end{array}\right.$$
    
    It remains to deal with the case $e'=f+1$.  Define $g_{e'}=h_{e'}$.
    There is a unique box $b'$ 
    such that $g_{e'}(b')$ is the box $\singlebox f$ in $C(e,f)_n$.
    Let $b''$ be the box $\singlebox{e'}$ in $C(e',f')_{n'}$.
    For a box $b$ with entry $e'$ define
    $$\widetilde g_{e'}(b)=\left\{\begin{array}{ll}
    \text{the box $\singlebox f$ in $C(e,f)_{n'}$} & \text{if $b=b'$} \\
    g_{e'}(b'') & \text{if $b$ is the box $\singlebox {e'}$ in $C(e',f')_n$}\\
    g_{e'}(b) & \text{otherwise.}
    \end{array}\right.$$

    The (EBP) is satisfied for the maps $h$, $\widetilde h$ by assumption.
    In our construction we have been careful to make sure that the (EBP)
    also holds for the modified maps $g$ and $\widetilde g$.
    It is easy to see that the maps $g$, $\widetilde g$ differ in exactly two orbits.
    The direct sum decompositions corresponding to $g$ and $\widetilde g$ given by
    Theorem~\ref{theorem-combinatorial-classification} hence differ, up to isomorphy
    and reordering, in exactly two poles and perhaps in the empty embeddings.
    (See the proof of Theorem ~\ref{theorem-combinatorial-classification}).
\end{proof}

\begin{ex}
In the above example, the poles are given by the orbits of the maps
$g$ and $\tilde g$:
For $\Gamma$, the orbit $\boxwithrow25\to \boxwithrow14$ yields the pole
$P=P((3,4))$, the orbit $\boxwithrow44\to\boxwithrow33\to\boxwithrow22\to\boxwithrow11$
yields $P'=P((0,1,2,3))$.  $\Gamma$ is the tableau for
the direct sum of those two poles and the empty embedding $E_{(2)}$.

\smallskip
For $\widetilde\Gamma$, the orbit $\boxwithrow24\to\boxwithrow13$ yields
$\widetilde P=P((2,3))$, while the long orbit gives rise to
$\widetilde P'=P((0,1,3,4))$.
The tableau for the direct sum of these two poles is $\widetilde\Gamma$.

\smallskip
Note that there are monomorphisms
$\widetilde P\to P$ and $P'\oplus E_{(2)}\to \widetilde P'$ which
both have cokernel $E_{(1)}$.
The monomorphisms are studied in more detail in the next
  section.

  $$
  \raisebox{.5cm}{$\Gamma:$} \quad
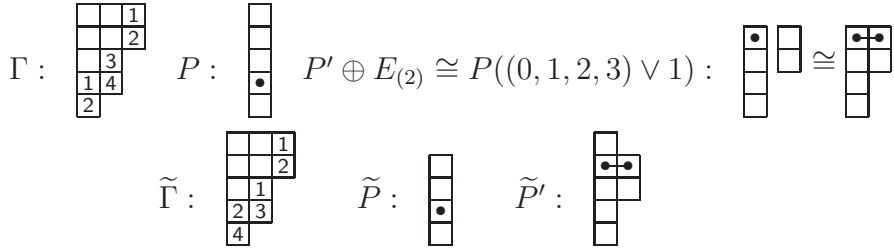
  \begin{picture}(9,15)
    \multiput(0,0)(0,3)5{\smbox}
    \multiput(3,3)(0,3)4{\smbox}
    \multiput(6,9)(0,3)2{\smbox}
    \put(6,12){\numbox1}
    \put(6,9){\numbox2}
    \put(3,6){\numbox3}
    \put(3,3){\numbox4}
    \put(0,3){\numbox1}
    \put(0,0){\numbox2}
  \end{picture}
  \quad
  \raisebox{.5cm}{$P:$} \quad
  \begin{picture}(3,15)
    \multiput(0,0)(0,3)5{\smbox}
    \put(1.5,4.5)\sbullet
  \end{picture}
  \quad
  \raisebox{.5cm}{$P'\oplus E_{(2)}\cong P((0,1,2,3)\vee 1):$} \quad
  \begin{picture}(19.5,15)
    \multiput(0,0)(0,3)4{\smbox}
    \put(1.5,10.5)\sbullet
    \multiput(4.5,6)(0,3)2{\smbox}
    \put(9,6){$\cong$}
    \multiput(13.5,0)(0,3)4{\smbox}
    \multiput(16.5,6)(0,3)2{\smbox}
    \multiput(15,10.5)(3,0)2\sbullet
    \put(15,10.5){\line(1,0)3}
  \end{picture}
  $$
  $$
  \raisebox{.5cm}{$\widetilde\Gamma:$} \quad
  \begin{picture}(9,15)
    \multiput(0,0)(0,3)5{\smbox}
    \multiput(3,3)(0,3)4{\smbox}
    \multiput(6,9)(0,3)2{\smbox}
    \put(6,12){\numbox1}
    \put(6,9){\numbox2}
    \put(3,6){\numbox1}
    \put(3,3){\numbox3}
    \put(0,3){\numbox2}
    \put(0,0){\numbox4}
  \end{picture}
  \qquad
  \raisebox{.5cm}{$\widetilde P:$} \quad
  \begin{picture}(3,15)
    \multiput(0,0)(0,3)4{\smbox}
    \put(1.5,4.5)\sbullet
  \end{picture}
  \qquad
  \raisebox{.5cm}{$\widetilde P':$} \quad
  \begin{picture}(19.5,15)
    \multiput(0,0)(0,3)5{\smbox}
    \multiput(3,6)(0,3)2{\smbox}
    \multiput(1.5,10.5)(3,0)2\sbullet
    \put(1.5,10.5){\line(1,0)3}
  \end{picture}
  $$

\end{ex}

%============================================================================
\section{An extension of two cyclic embeddings}
%============================================================================
\mylabel{sec-two-mono}

  A increasing box move on the LR-tableau $\Gamma$
  gives rise to two monomorphisms with
  isomorphic cokernels.  The main result states that
  the pull-back of the two cokernel maps
  yields an embedding $Q$ with LR-tableau $\Gamma$.
This prepares our study of some geometric properties of 
invariant subspace varieties in Section~\ref{sec-boundary}.

%==============================================================
\subsection{Two monomorphisms}
%==============================================================

Suppose that the LR-tableau $\widetilde\Gamma$
is obtained from $\Gamma$ by an increasing box move.
By the definition, there are
natural numbers $n>n'$, $e<e'$, $f<f'$ with $f-e=f'-e'$ such that
$\Gamma$, $\widetilde\Gamma$ are both unions of columns, such that 
$\widetilde\Gamma$ is obtained from $\Gamma$ by modifying two columns:
\begin{eqnarray*} \text{In $\Gamma$:} & C=C(e,f)_n, & C'=C(e',f')_{n'};\\
\text{in $\widetilde\Gamma$:} & \widetilde C=C(e,f)_{n'}, & 
  \widetilde C'=C(e',f')_n\end{eqnarray*}

Suppose that the columns are parts of the LR-tableaux 
of (possibly extended) poles $R$, $R'$, $\widetilde R$, $\widetilde R'$,
respectively.  We show in this section that there are monomorphisms
$\widetilde R\to R$, $R'\to \widetilde R'$ which both have cokernel 
$E_{(n-n')}$.

\medskip
More precisely, in the proof of Proposition~\ref{prop-simul-pole-decomp}
we have constructed partial maps $g$, $\widetilde g$, which differ
in exactly the two orbits which correspond to those columns.  
Hence there are poles with the following properties:
\begin{align*}
  P &= P((m_i)),  &    m_{e-1} &= n-f+e,\ldots, m_{f-1}=n \\
  P' &= P((m'_i)), &  m'_{e'-1} &= n'-f+e, \ldots, m'_{f'-1}=n'\\
  \widetilde P &= P((\widetilde m_i)),
  & \widetilde m_{e-1}&= n'-f+e,\ldots,\widetilde m_{f-1}=n'\\
  & & & \text{and}\; \widetilde m_i=m_i \;\text{for}\; i\not\in \{e-1,\ldots,f-1\}\\
  \widetilde P' &= P((\widetilde m'_i)),
  & \widetilde m'_{e'-1} &= n-f+e,\ldots,\widetilde m'_{f'-1}=n\\
  & & & \;\text{and}\; \widetilde m'_i=m'_i\; \text{for}\; i\not\in\{e'-1,\ldots,f'-1\}
\end{align*}

\smallskip
Comparing the height sequences $(m_i)$, $(\widetilde m_i)$,
we see that $\widetilde m_i=m_i-n+n'$ for $i\in\{e-1,\ldots,f-1\}$
but $\widetilde m_i=m_i$ otherwise.  As $n>n'$ then, whenever $e>2$, the sequence
$(m_i)$ has a gap after $m_{e-2}$ while the sequence $(\widetilde m_i)$
has a gap only if $\widetilde m_{e-1}\neq\widetilde m_{e-2}+1$.
Hence, if $e>2$ and $\widetilde m_{e-1}=\widetilde m_{e-2}+1$,
the ambient space for $P$ contains a summand $N_{(\widetilde m_{e-1})}$ which
is missing in $\widetilde P$.  In this case, we add this missing
summand back in: $\widetilde R=\widetilde P\oplus E_{(\widetilde m_{e-1})}$.
By Lemma~\ref{lemma-non-gap}, $\widetilde R\cong 
P((\widetilde m_i)\vee \widetilde m_{e-2})$.
Similarly, there may be missing gaps after $m_{f-1}$ for the sequence
$(m_i)$ when compared with $(\widetilde m_i)$; after $\widetilde m'_{f'-1}$ for
the sequence $(\widetilde m'_i)$ when compared to $(m'_i)$;
and after $m'_{e'-2}$ for the sequence $(m'_i)$ when compared to
$(\widetilde m'_i)$.  Formally we define:
\begin{align*}
  R &= \left\{\begin{array}{ll} P\oplus E_{(m_{f-1})} \cong
  P((m_i)\vee m_{f-1}) & \text{if} \; m_f=m_{f-1}+1 \\
  P & \text{otherwise}\end{array}\right. \\
  R' &= \left\{\begin{array}{ll} P'\oplus E_{(m'_{e'-2})} \cong
  P((m'_i)\vee m'_{e'-2}) & \text{if} \; e'\geq 2 \\
  & \quad \;\text{and}\; m_{e'-1}=m_{e'-2}+1 \\
  P' & \text{otherwise}\end{array}\right. \\
  \widetilde R &= \left\{\begin{array}{ll} \widetilde P\oplus E_{(\widetilde m_{e-2})} \cong
  P((\widetilde m_i)\vee \widetilde m_{e-2}) & \text{if} \; e\geq 2 \;
  \text{and}\; \widetilde m_{e-1}=\widetilde m_{e-2}+1 \\
  \widetilde P & \text{otherwise}\end{array}\right. \\
  \widetilde R' &= \left\{\begin{array}{ll} \widetilde P'\oplus E_{(\widetilde m'_{f'-1})} \cong
  P((\widetilde m'_i)\vee \widetilde m'_{f'-1}) & \text{if} \; \widetilde m_{f'}=\widetilde m_{f'-1}+1 \\
  \widetilde P' & \text{otherwise}\end{array}\right. \\
\end{align*}

The monomorphisms are as follows.
Write $R=((a)\subset B)$ where $B=N_\beta$
and $a=\sum p^{\ell_i} b_{\beta_i}$.
Then $\beta$ has a unique part $n$, say $\beta_u=n$, and $\ell_u=n-f$
(since $f$ is minimal with respect to the property 
that $p^f\cdot p^{\ell_u} b_n=0$).

\smallskip
Then $\widetilde R=((a)\subset \widetilde B)$ 
is the submodule of $R$ the ambient space
of which is generated by 
$b_{\beta_1},\ldots,b_{\beta_{u-1}},p^{n-n'}b_{\beta_u}, b_{\beta_{u+1}},\ldots$.
Thus, $\widetilde B=N_{\tilde \beta}$, where $\tilde \beta$ is 
obtained from $\beta$ replacing $\beta_u=n$ by
$\tilde \beta_u=n'$.  Hence $a=\sum p^{\tilde\ell_i} b_{\tilde\beta_i}$
where $\tilde\ell_i=\ell_i$ for $i\neq u$ and $\tilde\ell_u=n'-f$.
Clearly, $R/\widetilde R\cong E_{(n-n')}$.

\smallskip
Similarly, write $\widetilde R'=((c)\subset \widetilde D)$.
If $\widetilde D=N_{\tilde\delta}$ then $\tilde\delta$ 
contains a unique part $n$, say $\tilde\delta_v=n$.
Write $c=\sum p^{\tilde k_i} d_{\tilde\delta_i}$, then $\tilde k_v=n-f'$, as 
above.  

\smallskip
\sloppypar
The submodule $R'$ of $\widetilde R'$ has ambient space $D$ generated by 
$d_{\tilde\delta_1},\cdots,d_{\tilde\delta_{v-1}},p^{n-n'}d_{\tilde\delta_v},
d_{\tilde\delta{v+1}},\ldots$, so $D\cong N_\delta$ where the partition
$\delta$ is obtained from $\tilde\delta$ by replacing $\tilde\delta_v=n$
by $\delta_v=n'$.  We write $c=\sum p^{k_i}d_{\delta_i}$ where 
$k_v=n'-f'$ and $k_i=\tilde k_i$ for $i\neq v$.
Clearly, $\widetilde R'/R'\cong E_{(n-n')}$.

We summarize:

\begin{prop}\mylabel{prop-monomorphisms}
  Suppose $\Gamma, \widetilde\Gamma$ are LR-tableaux
  such that $\widetilde \Gamma$ is obtained from $\Gamma$ 
  by an increasing box move. 
  \begin{enumerate}\item
    Then there are cyclic embeddings
    $R$, $\widetilde R$, $R'$, $\widetilde R'$ and an embedding $S$,
    such that $S\oplus R\oplus \widetilde R$ has tableau $\Gamma$
    and $S\oplus R'\oplus\widetilde R'$ has tableau $\widetilde \Gamma$.
  \item
    There are monomorphisms of embeddings 
    $\widetilde R\to R$ and $R'\to \widetilde R'$
    such that both maps have cokernel $E_{(n-n')}$,
    where $n$ and $n'$ are the lengths of the columns in which
    the entries are exchanged.
  \end{enumerate}
  \qed
\end{prop}

%============================================================================
\subsection{An extension of two cyclic embeddings}
%============================================================================
\mylabel{sec-extension}

In the set-up from the previous section,
we show that there is an embedding $Q$ and two exact sequences of embeddings,
each with $Q$ as the middle term, such that the sum of the end terms
of the two exact sequences have tableaux
$\Gamma$ and $\widetilde\Gamma$, respectively.
The main result is that  $Q$ has LR-tableau $\Gamma$.

\smallskip
We use the cokernel maps from
Proposition~\ref{prop-monomorphisms} to construct $Q$ as a pullback
in the category of homomorphisms of $\Lambda$-modules.

$$\begin{CD}
  @.  @.  0  @. 0 \\
  @. @.  @VVV @VVV \\
  @.      @.   \widetilde R  @=  \widetilde R  \\
  @.      @.   @VVV     @VVV \\
  0 @>>> R' @>>> Q @>>> R @>>> 0 \\
  @. @| @VVV @VVV \\
  0 @>>> R' @>>> \widetilde R' @>>> E_{(n-n')} @>>> 0 \\
  @. @. @VVV @VVV \\
  @. @. 0 @. 0
\end{CD}
$$

In general, as in the following
example, the middle row is not split exact:

\begin{ex}
Consider the following two tableaux from the example in the Introduction.
  $$
  \raisebox{.5cm}{$\Gamma=\Gamma_2:$} \quad
  \begin{picture}(15,15)
    \multiput(0,3)(0,3)4{\smbox}
    \multiput(3,6)(0,3)3{\smbox}
    \multiput(6,9)(0,3)2{\smbox}
    \multiput(9,12)(0,3)1{\smbox}
    \put(12,12){\numbox1}
    \put(9,9){\numbox2}
    \put(6,6){\numbox1}
    \put(3,3){\numbox3}
    \put(0,0){\numbox2}
  \end{picture}
\qquad
  \raisebox{.5cm}{$\widetilde\Gamma=\Gamma_{3a}:$} \quad
  \begin{picture}(15,15)
    \multiput(0,3)(0,3)4{\smbox}
    \multiput(3,6)(0,3)3{\smbox}
    \multiput(6,9)(0,3)2{\smbox}
    \multiput(9,12)(0,3)1{\smbox}
    \put(12,12){\numbox1}
    \put(9,9){\numbox1}
    \put(6,6){\numbox2}
    \put(3,3){\numbox3}
    \put(0,0){\numbox2}
  \end{picture}
$$

$\widetilde\Gamma$ is obtained
from $\Gamma$ by an increasing box move.  It gives rise to the 
following pole decomposition.
  $$
  \raisebox{.5cm}{For $\Gamma,\; R=P((2,4)):$} \;
  \begin{picture}(6,15)
    \multiput(0,0)(0,3)5{\smbox}
    \multiput(3,3)(0,3)3{\smbox}
    \multiput(1.5,4.5)(3,0)2{\sbullet}
    \put(1.5,4.5){\line(1,0){3}}
  \end{picture}
\quad
  \raisebox{.5cm}{$R'=P((0,1,3)\vee 0):$} \;
  \begin{picture}(9,15)
    \multiput(0,0)(0,3)4{\smbox}
    \multiput(3,3)(0,3)2{\smbox}
    \multiput(6,6)(0,3)1{\smbox}
    \multiput(1.5,7.5)(3,0)3{\sbullet}
    \put(1.5,7.5){\line(1,0)6}
  \end{picture}
\;
  \raisebox{.5cm}{$\cong$} \;
  \begin{picture}(12,15)
    \multiput(0,0)(0,3)4{\smbox}
    \multiput(3,3)(0,3)2{\smbox}
    \multiput(9,6)(0,3)1{\smbox}
    \multiput(1.5,7.5)(3,0)2{\sbullet}
    \put(1.5,7.5){\line(1,0)3}
  \end{picture}
$$
  $$
  \raisebox{.5cm}{For $\widetilde\Gamma,\; \widetilde R=P((1,4)):$} \;
  \begin{picture}(6,15)
    \multiput(0,0)(0,3)5{\smbox}
    \multiput(3,3)(0,3)2{\smbox}
    \multiput(1.5,4.5)(3,0)2{\sbullet}
    \put(1.5,4.5){\line(1,0){3}}
  \end{picture}
\quad
  \raisebox{.5cm}{$\widetilde R'=P((0,2,3)\vee 2):$} \;
  \begin{picture}(9,15)
    \multiput(0,0)(0,3)4{\smbox}
    \multiput(3,3)(0,3)3{\smbox}
    \multiput(6,6)(0,3)1{\smbox}
    \multiput(1.5,7.5)(3,0)3{\sbullet}
    \put(1.5,7.5){\line(1,0)6}
  \end{picture}
\;
  \raisebox{.5cm}{$\cong$} \;
  \begin{picture}(12,15)
    \multiput(0,0)(0,3)4{\smbox}
    \multiput(3,6)(0,3)1{\smbox}
    \multiput(9,3)(0,3)3{\smbox}
    \multiput(1.5,7.5)(3,0)2{\sbullet}
    \put(1.5,7.5){\line(1,0)3}
  \end{picture}
$$
Note that the embeddings $\widetilde R\to R$ and $R'\to \widetilde R'$
both have cokernel $E_{(1)}$.   The pullback $Q$ of the diagram
$\begin{CD} @. R\\ @. @VVV\\ \widetilde R' @>>> E_{(1)}\end{CD}$
has the following direct sum decomposition.
$$
\raisebox{.5cm}{$Q:$} \quad
\begin{picture}(15,15)
  \multiput(0,0)(0,3)5{\smbox}
  \multiput(3,3)(0,3)3{\smbox}
  \multiput(6,3)(0,3)2{\smbox}
  \multiput(9,0)(0,3)4{\smbox}
  \multiput(12,6)(0,3)1{\smbox}
  \multiput(1.5,4.5)(3,0)3{\sbullet}
  \multiput(7.5,7.5)(3,0)3{\sbullet}
  \put(1.5,4.5){\line(1,0){6}}
  \put(7.5,7.5){\line(1,0){6}}
\end{picture}
\quad
\raisebox{.5cm}{$\cong$} \quad
\begin{picture}(18,15)
  \multiput(0,0)(0,3)5{\smbox}
  \multiput(3,3)(0,3)3{\smbox}
  \multiput(6,3)(0,3)2{\smbox}
  \multiput(9,0)(0,3)4{\smbox}
  \multiput(15,6)(0,3)1{\smbox}
  \multiput(1.5,4.5)(3,0)3{\sbullet}
  \multiput(7.5,7.5)(3,0)2{\sbullet}
  \put(1.5,4.5){\line(1,0){6}}
  \put(7.5,7.5){\line(1,0){3}}
\end{picture}
\quad
\raisebox{.5cm}{$\cong$} \quad
\begin{picture}(18,15)
  \multiput(0,0)(0,3)5{\smbox}
  \multiput(3,3)(0,3)3{\smbox}
  \multiput(6,3)(0,3)2{\smbox}
  \multiput(9,0)(0,3)4{\smbox}
  \multiput(15,6)(0,3)1{\smbox}
  \multiput(1.5,4.5)(6,0)2{\sbullet}
  \multiput(4.5,7.5)(3,0)3{\sbullet}
  \put(1.5,4.5){\line(1,0){6}}
  \put(4.5,7.5){\line(1,0){6}}
\end{picture}
\quad
\raisebox{.5cm}{$\cong$} \quad
\begin{picture}(21,15)
  \multiput(0,0)(0,3)5{\smbox}
  \multiput(3,3)(0,3)2{\smbox}
  \multiput(6,0)(0,3)4{\smbox}
  \multiput(12,3)(0,3)3{\smbox}
  \multiput(18,6)(0,3)1{\smbox}
  \multiput(1.5,4.5)(3,0)2{\sbullet}
  \multiput(4.5,7.5)(3,0)2{\sbullet}
  \put(1.5,4.5){\line(1,0){3}}
  \put(4.5,7.5){\line(1,0){3}}
\end{picture}
$$
Each isomorphism is given by a change of the generators of the ambient space:
$b_2'=b_2+b_1$; $b_2''=b_2'+pb_3$, $b_3''=-b_3$; $b_4'''=b_4+b_3''$.
The last sum contains a summand which is an indecomposable embedding which 
is not cyclic, see \cite[(6.5)]{rs}.
\end{ex}

We define the module $Q$ and the maps in the pushout diagram.
For this, we use the notation from above Proposition~\ref{prop-monomorphisms},
in particular, $R=((a)\subset B)$ where $a=\sum p^{\ell_i}b_{\beta_i}$
and $R'=((c)\subset D)$ where $c=\sum p^{k_i}d_{\delta_i}$. 

\smallskip
Put $Q=((r,s)\subset B\oplus D)$ where $r=(a,p^{n'-f}d_{n'})$
and $s=(0,c)$.

\smallskip
Clearly, $R'$ embeds into the second component in $Q$; the 
cokernel of this embedding is $R$.

\smallskip
The module $\widetilde R$ embeds into the first component and the 
summand of the second component generated by $d_{n'}$; 
for $i\neq u$, $\tilde b_{\beta_i}$ is mapped to $b_{\beta_i}$ while
$\tilde b_{\beta_u}=b_{n'}$ is mapped to $(p^{n-n'} b_n,d_{n'})$.
In particular, $p^{\tilde\ell_u} b_{n'}=p^{n'-f} b_{n'}$ maps to 
$(p^{n-f}b_n,p^{n'-f}d_{n'})=(p^{\ell_u}b_n,p^{n'-f}d_{n'})$,
so $a\in\widetilde B$ is sent to $(a,p^{n'-f}d_{n'})=r\in Q$.
The cokernel of this map is $Q\to \widetilde R'$; here the 
second component $D$ of the ambient space is included into $\widetilde D$,
as in the embedding $R'\to \widetilde R'$ above,
and the map $B\to \widetilde D$ is given by $b_{\beta_i}\mapsto 0$
if $i\neq u$ and $b_{\beta_u}=b_n\mapsto -d_n$.

\smallskip
With these modules and maps, all the squares in the pullback diagram
are commutative.

\begin{prop}\mylabel{prop-same-tableau}
The embeddings $Q$, $R\oplus R'$ have the same tableau.
\end{prop}

\begin{proof}
Suppose $R\oplus R'$ has tableau $\Gamma=[\gamma^{(\ell)}]_\ell$
and $Q$ has tableau $\Delta=[\delta^{(\ell)}]_\ell$.
The two tableaux are equal if we can show 
for each pair of natural numbers $\ell$, $m$ that the tableaux 
$\Gamma$, $\Delta$ have the same number
of boxes which are empty or have entry at most $\ell$ in the first $m$ rows:
$$(\gamma^{(\ell)})'_1+\cdots+(\gamma^{(\ell)})'_m=(\delta^{(\ell)})'_1+
   \cdots(\delta^{(\ell)})'_m.$$

We denote by $H$ the common ambient space of $R\oplus R'=
((a,c)\subset B\oplus D)$ and of $Q=((r,s)\subset B\oplus D)$,
and denote the two subspaces by $U$ and $W$, respectively.
If the embedding $(U\subset H)$ has tableau $\Gamma$, 
the embedding $((U+p^mH)/p^mH\subset H/p^mH)$ has the tableau which
consists of the first $m$ rows of $\Gamma$.
The number of boxes which are empty or have entry at most $\ell$,
given by the left hand side of the above equation, is just the length
of the module
$$\big(H/p^m H\big) \big/ \big ((p^\ell U +p^m H)/(p^mH)\big)
\cong H/(p^\ell U+p^m H).$$

Let $(\mathcal F_\alpha)$ and $(\mathcal G_\alpha)$ be the filtrations of 
$B$ and $D$ centered at the elements $a$ and $c$, respectively,
see Lemma~\ref{lemma-filtration}.
Consider the submodule generators for $Q$, 
$r=(a,p^{n'-f}d_{n'})$ and $s=(0,c)$.
We define a filtration $(\mathcal H_\alpha)$ for $H=B\oplus D$ by 
adjusting the filtrations $(\mathcal F_\alpha)$ and $(\mathcal G_\alpha)$ 
for $B$ and $D$ such that 
the elements $r$ and $s$ are homogeneous in degrees $f$ and $f'$,
respectively:
$$\mathcal H_\alpha=\mathcal F_{\alpha-f}\oplus \mathcal G_{\alpha-f'}.$$

We have seen above that we need to show that for all $\ell, m\in\mathbb N$,
the $\Lambda$-modules $p^\ell U+p^m H$ and $p^\ell W+p^m H$ have the same 
length.  For this, it suffices to show that for each $\alpha\in\mathbb Z$,
$\ell,m\in\mathbb N$, the subspaces 
$U_{\alpha,\ell,m}$ and $W_{\alpha,\ell,m}$ of 
$V_\alpha=\mathcal H_\alpha/\mathcal H_{\alpha-1}$ have the same dimension:
\begin{eqnarray*}
  U_{\alpha,\ell,m} &=& \frac{[(p^\ell U+p^m H)\cap \mathcal H_\alpha]+H_{\alpha-1}}{%
\mathcal H_{\alpha-1}};\\
W_{\alpha,\ell,m} &=& \frac{[(p^\ell W+p^m H)\cap \mathcal H_\alpha]+H_{\alpha-1}}{%
\mathcal H_{\alpha-1}}\end{eqnarray*}

This statement holds for $\alpha>f-\ell$ and for $\alpha\leq 0$ since 
$U_{\alpha,\ell,m}=W_{\alpha,\ell,m}$.  

\smallskip
Two cases remain:

\smallskip
Assume $1\leq\alpha\leq\max\{f-\ell,1+f-e\}$.
Then $V_\alpha$ has dimension $u+v$, a basis is induced by the generators
of $B$ and $D$:  
$b_{\beta_1},\ldots,b_{\beta_u}=b_n; d_{\delta_1},\ldots d_{\delta_v}=d_{n'}.$
Here is a sketch of the basis for $V_\alpha$ together with the 
subspace generators for $U_{\alpha,0,0}$ and 
$W_{\alpha,0,0}$:
$$
\raisebox{0.4cm}{$U_{\alpha,0,0}:$} \quad
\begin{picture}(30,10)(0,-3)
  \put(0,3){\smbox}
  \put(0,0){$b_{\beta_1}$}
  \put(3,3){\smbox}
  \put(4,-3){$b_{\beta_2}$}
  \put(7,3){$\cdots$}
  \put(12,3){\smbox}
  \put(12,0){$b_n$}
  \put(15,3){\smbox}
  \put(15,-3){$d_{n'}$}
  \put(18,3){\smbox}
  \put(18,0){$d_{\delta_{v-1}}$}
  \put(22,3){$\cdots$}
  \put(27,3){\smbox}
  \put(27,0){$d_{\delta_1}$}
  \multiput(1.5,4.5)(3,0)2{\sbullet}
  \multiput(13.5,4.5)(3,0)3{\sbullet}
  \multiput(28.5,4.5)(3,0)1{\sbullet}
  \put(1.5,4.5){\line(1,0){5.5}}
  \put(11,4.5){\line(1,0){2.5}}
  \put(16.5,4.5){\line(1,0){5.5}}
  \put(26,4.5){\line(1,0){2.5}}
\end{picture}
\qquad
\raisebox{0.4cm}{$W_{\alpha,0,0}:$} \quad
\begin{picture}(30,10)(0,-3)
  \put(0,3){\smbox}
  \put(0,0){$b_{\beta_1}$}
  \put(3,3){\smbox}
  \put(4,-3){$b_{\beta_2}$}
  \put(7,3){$\cdots$}
  \put(12,3){\smbox}
  \put(12,0){$b_n$}
  \put(15,3){\smbox}
  \put(15,-3){$d_{n'}$}
  \put(18,3){\smbox}
  \put(18,0){$d_{\delta_{v-1}}$}
  \put(22,3){$\cdots$}
  \put(27,3){\smbox}
  \put(27,0){$d_{\delta_1}$}
  \multiput(1,5)(3,0)2{\sbullet}
  \multiput(13,5)(3,0)2{\sbullet}
  \multiput(17,4)(3,0)2{\sbullet}
  \multiput(29,4)(3,0)1{\sbullet}
  \put(1,5){\line(1,0){6}}
  \put(11,5){\line(1,0){5}}
  \put(17,4){\line(1,0){5}}
  \put(26,4){\line(1,0){3}}
\end{picture}
$$
If $\alpha<n'-m$ then both $U_{\alpha,\ell,m}$ and $W_{\alpha,\ell,m}$
have dimension two, as there is no contribution from the term $p^mH$.
Otherwise, the basic element corresponding to $d_{n'}$ is contained 
in the space generated by $p^mH$, and hence $U_{\alpha,\ell,m}=W_{\alpha,\ell,m}$.

\smallskip
It remains to consider the case where $\max\{f-l,1+f-e\}<\alpha\leq f-\ell$
(this case does not occur if $e=1$).
Here, the space $V_\alpha$ contains at least two additional basic elements,
given by $b_{\beta_{u+1}}$ and $d_{\delta_{v+1}}$. 
Again, if $\alpha<n'-m$, then both $U_{\alpha,\ell,m}$ and $W_{\alpha,\ell,m}$
have the same dimension 
$2+\dim[(p^mH\cap \mathcal H_\alpha)+\mathcal H_{\alpha-1}]/\mathcal H_{\alpha-1}$;
otherwise the basic element corresponding to $d_{n'}$
is contained in the space generated by $p^mH$, and then
the subspaces are equal:  $U_{\alpha,\ell,m}=W_{\alpha,\ell,m}$.

\smallskip
We have seen that for all choices of $\alpha$, $\ell$ and $m$, the
spaces $U_{\alpha,\ell,m}$, $W_{\alpha,\ell,m}$ have the same dimension;
it follows that $Q$ and $R\oplus R'$ have the same tableau:
$\Delta=\Gamma$.
\end{proof}

The following example is a minor modification of the previous and 
may help to follow the details in the proof.

\begin{ex} 
Consider the LR-tableaux:
  $$
  \raisebox{.5cm}{$\Gamma:$} \quad
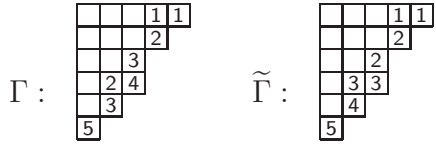
  \begin{picture}(15,18)
    \multiput(0,3)(0,3)5{\smbox}
    \multiput(3,9)(0,3)3{\smbox}
    \multiput(6,12)(0,3)2{\smbox}
    \put(12,15){\numbox1}
    \put(9,15){\numbox1}
    \put(9,12){\numbox2}
    \put(6,9){\numbox3}
    \put(6,6){\numbox4}
    \put(3,6){\numbox2}
    \put(3,3){\numbox3}
    \put(0,0){\numbox5}
  \end{picture}
\qquad
  \raisebox{.5cm}{$\widetilde\Gamma:$} \quad
  \begin{picture}(15,18)
    \multiput(0,3)(0,3)5{\smbox}
    \multiput(3,9)(0,3)3{\smbox}
    \multiput(6,12)(0,3)2{\smbox}
    \put(12,15){\numbox1}
    \put(9,15){\numbox1}
    \put(9,12){\numbox2}
    \put(6,9){\numbox2}
    \put(6,6){\numbox3}
    \put(3,6){\numbox3}
    \put(3,3){\numbox4}
    \put(0,0){\numbox5}
  \end{picture}
$$
$\widetilde\Gamma$ is obtained
from $\Gamma$ by an increasing box move which exchanges 
the columns $C(2,3)_5$ and $C(3,4)_4$ in $\Gamma$ by
$C(3,4)_5$ and $C(2,3)_4$.  It gives rise to the 
following pole decomposition: 
$R=P((0,3,4))$, $R'=P((0,1,2,3,5)\vee 1)$,
$\widetilde R=P((0,2,3))$, $\widetilde R'=P((0,1,3,4,5)\vee 4)$.
  $$
  \raisebox{.8cm}{$R:$} \quad
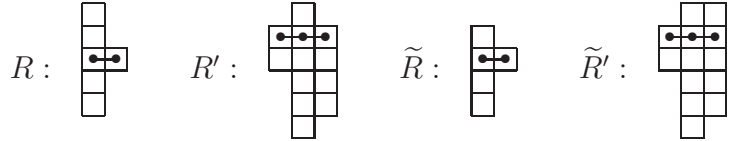
  \begin{picture}(6,18)
    \multiput(0,3)(0,3)5{\smbox}
    \multiput(3,9)(0,3)1{\smbox}
    \multiput(1.5,10.5)(3,0)2{\sbullet}
    \put(1.5,10.5){\line(1,0){3}}
  \end{picture}
\qquad
  \raisebox{.8cm}{$R':$} \quad
  \begin{picture}(9,18)
    \multiput(0,9)(0,3)2{\smbox}
    \multiput(3,0)(0,3)6{\smbox}
    \multiput(6,3)(0,3)4{\smbox}
    \multiput(1.5,13.5)(3,0)3{\sbullet}
    \put(1.5,13.5){\line(1,0){6}}
  \end{picture}
\qquad
  \raisebox{.8cm}{$\widetilde R:$} \quad
  \begin{picture}(6,18)
    \multiput(0,3)(0,3)4{\smbox}
    \multiput(3,9)(0,3)1{\smbox}
    \multiput(1.5,10.5)(3,0)2{\sbullet}
    \put(1.5,10.5){\line(1,0){3}}
  \end{picture}
\qquad
  \raisebox{.8cm}{$\widetilde R':$} \quad
  \begin{picture}(9,18)
    \multiput(0,9)(0,3)2{\smbox}
    \multiput(3,0)(0,3)6{\smbox}
    \multiput(6,3)(0,3)5{\smbox}
    \multiput(1.5,13.5)(3,0)3{\sbullet}
    \put(1.5,13.5){\line(1,0){6}}
  \end{picture}
$$

In the proof we show that $Q$ and $R\oplus R'$ have the same tableau.
The columns in the diagrams are labeled by the generators of $H_\Lambda$;
the boxes in a row labeled $\alpha$ correspond to a basis for $V_\alpha$.
Note that the submodule generators $r$, $s$ of $Q$ and 
$(a,0)$ and $(0,c)$ of $R\oplus R'$ are homogeneous elements in
degrees $f$ and $f'$, respectively.

$$
  \raisebox{1.1cm}{$Q:$} \quad
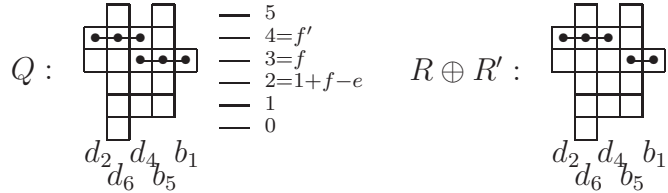
  \begin{picture}(35,18)(0,-3)
    \multiput(0,9)(0,3)2{\smbox}
    \multiput(3,0)(0,3)6{\smbox}
    \multiput(6,3)(0,3)4{\smbox}
    \multiput(9,3)(0,3)5{\smbox}
    \multiput(12,9)(0,3)1{\smbox}
    \multiput(1.5,13.5)(3,0)3{\sbullet}
    \multiput(7.5,10.5)(3,0)3{\sbullet}
    \put(1.5,13.5){\line(1,0){6}}
    \put(7.5,10.5){\line(1,0){6}}
    \put(0,-3){$d_2$}
    \put(3,-6){$d_6$}
    \put(6,-3){$d_4$}
    \put(9,-6){$b_5$}
    \put(12,-3){$b_1$}
    \multiput(18,1.5)(0,3)6{\line(1,0){4}}
    \put(24,1){$\scriptstyle 0$}
    \put(24,4){$\scriptstyle 1$}
    \put(24,7){$\scriptstyle 2=1+f-e$}
    \put(24,10){$\scriptstyle 3=f$}
    \put(24,13){$\scriptstyle 4=f'$}
    \put(24,16){$\scriptstyle 5$}
  \end{picture}
\qquad
  \raisebox{1.1cm}{$R\oplus R':$} \quad
  \begin{picture}(15,18)(0,-3)
    \multiput(0,9)(0,3)2{\smbox}
    \multiput(3,0)(0,3)6{\smbox}
    \multiput(6,3)(0,3)4{\smbox}
    \multiput(9,3)(0,3)5{\smbox}
    \multiput(12,9)(0,3)1{\smbox}
    \multiput(1.5,13.5)(3,0)3{\sbullet}
    \multiput(10.5,10.5)(3,0)2{\sbullet}
    \put(1.5,13.5){\line(1,0){6}}
    \put(10.5,10.5){\line(1,0){3}}
    \put(0,-3){$d_2$}
    \put(3,-6){$d_6$}
    \put(6,-3){$d_4$}
    \put(9,-6){$b_5$}
    \put(12,-3){$b_1$}
  \end{picture}
$$
\end{ex}

%============================================================================
\section{The boundary relation for invariant subspace varieties}
%============================================================================
\mylabel{sec-boundary}

In this section we assume that the field $k$ is algebraically closed
and that the discrete valuation ring 
$\Lambda$ is a $k$-algebra, for example
the power series ring $\Lambda= k[[T]]$ or the localization of the polynomial
ring $\Lambda= k[T]_{(T)}$.
In each case, a finite dimensional $\Lambda$-module $B$ is
a nilpotent linear operator, determined up to isomorphy by the partition 
which lists the sizes of the Jordan blocks.  A cyclic embedding
$(A\subset B)$ consists also of a subspace $A\subset B$, invariant under
the action of the operator, and such that the induced action on $A$ gives rise
to at most one Jordan block.

\smallskip
Here is our main result:

\begin{thm} \mylabel{theorem-one-parameter-family}
  Suppose that $\Gamma$, $\widetilde\Gamma$ are LR-tableaux
    such that $\widetilde \Gamma$ is obtained from $\Gamma$
    by an increasing box move.  
    There is a one-parameter family of embeddings $Q(\mu)$
    such that $Q(\mu)$ has tableau $\Gamma$ for $\mu\neq 0$
    and tableau $\widetilde \Gamma$ for $\mu=0$.
\end{thm}

\begin{ex}
We continue the example 
from the beginning of Section~\ref{sec-extension}.
For $\mu\in k$, the module $Q(\mu)$ is defined as follows:
$$
\raisebox{.5cm}{$Q(\mu):$} \quad
\begin{picture}(15,15)
  \multiput(0,0)(0,3)5{\smbox}
  \multiput(3,3)(0,3)3{\smbox}
  \multiput(6,3)(0,3)2{\smbox}
  \multiput(9,0)(0,3)4{\smbox}
  \multiput(12,6)(0,3)1{\smbox}
  \multiput(1.5,4.5)(6,0)2{\sbullet}
  \multiput(4.5,7.5)(3,0)1{\sbullet}
  \multiput(10.5,7.5)(3,0)2{\sbullet}
  \put(1.5,4.5){\line(1,0){6}}
  \put(4.5,7.5){\line(1,0){2}}
  \put(8.5,7.5){\line(1,0){5}}
  \put(7.5,7.5){\makebox[0mm]{$\scriptstyle\mu$}}
  \put(0,-3){$b_5$}
  \put(3,-5){$b_3$}
  \put(6,-3){$d_2$}
  \put(9,-5){$d_4$}
  \put(12,-3){$d_1$}
\end{picture}
\qquad
\raisebox{.5cm}{$Q:$} \quad
\begin{picture}(15,15)
  \multiput(0,0)(0,3)5{\smbox}
  \multiput(3,3)(0,3)3{\smbox}
  \multiput(6,3)(0,3)2{\smbox}
  \multiput(9,0)(0,3)4{\smbox}
  \multiput(12,6)(0,3)1{\smbox}
  \multiput(1.5,4.5)(3,0)3{\sbullet}
  \multiput(7.5,7.5)(3,0)3{\sbullet}
  \put(1.5,4.5){\line(1,0){6}}
  \put(7.5,7.5){\line(1,0){6}}
\end{picture}
\qquad
  \raisebox{.5cm}{$\widetilde R\oplus\widetilde R':$} \quad
  \begin{picture}(6,15)
    \multiput(0,0)(0,3)5{\smbox}
    \multiput(3,3)(0,3)2{\smbox}
    \multiput(1.5,4.5)(3,0)2{\sbullet}
    \put(1.5,4.5){\line(1,0){3}}
  \end{picture}
\;\raisebox{.5cm}{$\oplus$} \;
  \begin{picture}(9,15)
    \multiput(0,0)(0,3)4{\smbox}
    \multiput(3,3)(0,3)3{\smbox}
    \multiput(6,6)(0,3)1{\smbox}
    \multiput(1.5,7.5)(3,0)3{\sbullet}
    \put(1.5,7.5){\line(1,0)6}
  \end{picture}
\raisebox{-.5cm}{$\phantom\oplus$}
$$

We have also pictured $Q$ and $\widetilde R\oplus \widetilde R'$.
Clearly, for $\mu=0$, $Q(\mu)\cong\widetilde R\oplus\widetilde R'$.

\smallskip
For $\mu\neq 0$, write $Q(\mu)$ as the embedding
$((x,y)\subset B\oplus D)$ where $B\oplus D$, 
as the ambient space of $R\oplus R'$,
is generated by $b_5$, $b_3$, $b_2$, $d_4$, $d_1$, as before, and 
$x=p^3 b_5+p d_2$, $y=p b_3+\mu d_2+ p d_4 +d_1$.
The successive substitutions in $Q(\mu)$:
$d_2'=pb_3+\mu d_2$ (so $y=d_2'+p d_4+d_1$, $x=p^3 b_5-p^2 b_3/\mu+d_2'/\mu$)
and $x''=\mu x$, $b_5''=\mu b_5$, $b_3''=-b_3$ 
(so $x''=p^3 b_5''+p^2b_3''+pd_2'$)
show that $Q(\mu)\cong Q$.
\end{ex}

\begin{proof}[Proof of Theorem~\ref{theorem-one-parameter-family}]
The proof will follow the above example.
We assume the $\widetilde\Gamma$ is obtained from $\Gamma$ by an increasing
box move and use the notation from Section~\ref{sec-two-mono}
for the cyclic embeddings $R$, $R'$, $\widetilde R$ and $\widetilde R'$.
Write $R=((a)\subset B)$ where $B=N_\beta$, $a=\sum p^{\ell_i}b_{\beta_i}$,
with $u$ such that $\beta_u=n$.  Similarly, 
$R'=((c)\subset D)$ where $D=N_\delta$, $c=\sum p^{k_i}d_{\delta_i}$, 
with $v$ such that $\delta_v=n'$.
Then $Q$ is the embedding $((r,s)\subset B\oplus D)$ where 
$r=(a,p^{n'-f}d_{n'})$ and $s=(0,c)$, 
as above Proposition~\ref{prop-same-tableau}.

\smallskip
Define $Q(\mu)=((x,y)\subset B\oplus D)$ where
$x=(a-p^{\ell_u}b_n,p^{n'-f}d_{n'})$, $y=(p^{k_v+n-n'}b_n,c+(\mu-1)p^{k_v}d_{n'})$.

\smallskip
For $\mu=0$, there is a decomposition of $Q(0)$ given by the decomposition
of the ambient space as 
$$B\oplus D=
\left(\bigoplus_{i\neq u} b_{\beta_i}\Lambda\oplus d_{\delta_v}\Lambda\right)\oplus
\left(b_{\beta_u}\Lambda\oplus \bigoplus_{i\neq v}d_{\delta_i}\Lambda\right).$$
Note that $x$ is in the first summand, $\widetilde B$ say, and
$\widetilde R\cong ((x)\subset \widetilde B)$.
Also $y$ is in the second summand, $\widetilde D$ say, and
$\widetilde R'\cong ((y)\subset \widetilde D)$. 

\smallskip
Let $\mu\neq 0$, it remains to show that $Q(\mu)\cong Q$.
Substitute $d_{n'}'=p^{n-n'}b_n+\mu d_{n'}$, 
$D'=\bigoplus_{i\neq v}d_{\delta_i}\Lambda\oplus d_{n'}'\Lambda$, this yields
$y=(0,(c-p^{k_v}d_{n'})+p^{k_v}d_{n'}')\in B\oplus D'$, 
$x=(a-p^{\ell_u}b_n-p^{n-f}/\mu b_n,p^{n'-f}/\mu d_{n'}')$.
Then substitute $x''=\mu\cdot x$, $b_{\beta_i}''=\mu\cdot b_{\beta_i}$ 
for $i\neq u$ and $b_n''=-b_n$, $a''=\sum p^{\ell_i}b_{\beta_i}''$
to obtain $x''=(a'',p^{n'-f}d_{n'}')$.  This shows
$$Q(\mu)=((x'',y)\subset B\oplus D')\cong ((r,s)\subset B\oplus D)=Q.$$
\end{proof}

\bigskip
We obtain the following consequence for the representation space
of invariant subspaces of nilpotent linear operators.
Let $\alpha$, $\beta$, $\gamma$ be partitions.
Then the subset $\mathbb V_{\alpha,\gamma}^\beta$ of the affine variety
$\Hom_k(N_\alpha,N_\beta)$ consisting of all $k[T]$-linear 
monomorphisms with cokernel 
isomorphic to $N_\gamma$ is constructible.
If $\mathbb V_\Gamma$ denotes the subset of all embeddings with
tableau $\Gamma$, then the irreducible components of 
$\mathbb V_{\alpha,\gamma}^\beta$ are given by the closures
$\overline{\mathbb V}_\Gamma$, where $\Gamma$ runs over all LR-tableaux of shape
$(\alpha,\beta,\gamma)$.

\medskip
For LR-tableaux $\Gamma$, $\widetilde\Gamma$ of the same shape we write
$\Gamma\preceq_{\sf boundary}\widetilde\Gamma$ if 
$\mathbb V_{\widetilde\Gamma}\cap\overline{\mathbb V}_\Gamma\neq\emptyset$.

\medskip
Theorem~\ref{theorem-one-parameter-family} yields the following 
combinatorial criterion for the boundary relation which we stated above 
as Theorem~\ref{theorem-box-boundary}.

\begin{cor}\mylabel{cor-box-boundary}
  Suppose $\Gamma$, $\widetilde\Gamma$ are LR-tableaux
    such that 
  $\widetilde \Gamma$ is obtained from $\Gamma$ by an increasing
box move.  Then $\Gamma\prec_{\sf boundary}\widetilde\Gamma$. \qed
\end{cor}

For LR-tableaux $\Gamma$, $\widetilde\Gamma$ of the same shape 
$(\alpha,\beta,\gamma)$, we have seen in \cite{ks-mfo} 
that the boundary relation implies
the dominance relation (which is given by the dominance relation for
the defining partitions).  
As a consequence, the reflexive and transitive closure of the boundary relation $\prec_{\sf boundary}$
defines a partial order $\leq_{\sf boundary}$ on the set 
of LR-tableaux of the given shape $(\alpha,\beta,\gamma)$.
This is the poset $\mathcal P_{\sf boundary}$.

\medskip
In case the skew diagram $\beta\setminus\gamma$
is a horizontal and vertical strip, the dominance relation is generated
by the increasing box moves (\cite{kst}). 
As a consequence we can describe the relation $\leq_{\sf boundary}$:

\begin{cor}[Corollary~\ref{cor-rook-strip}]
The reflexive and transitive closure
of the boundary relation
is a partial order on the set of all LR-tableaux
of a given shape $(\alpha,\beta,\gamma)$.  In case $\beta\setminus\gamma$ is 
a horizontal and vertical strip, the dominance order, the closure 
of the box order, and the closure of the boundary relation agree.
\qed
\end{cor}

The case where the induced operator on the subspace is $T^2$-bounded
has been studied in \cite{ks-hall,ks-survey}.  Here, the boundary relation
(and not only its closure) is equivalent to the dominance order
(\cite[Section~5.3]{ks-survey}).

\section{Acknowledgement and Dedication}

\begin{acknowledgement}
Part of this research has been obtained while the authors were visiting the 
Mathematische Forschungsinstitut in Oberwolfach.  They would like to thank the institute
for their hospitality and for providing an outstanding environment for mathematical
research. 
\end{acknowledgement}

\begin{dedication}
The authors wish to dedicate this paper to Professor Fred Richman. 
His talk in 1999 about subgroups of $p^5$-bounded abelian groups \cite{R5} 
has introduced the second named author to the Birkhoff Problem.  
It was the first of many delightful presentations of Professor Richman, 
which in turn have led to
numerous inspiring discussions with him.  Since several
years, the Birkhoff Problem and its many variations have been a major theme
in the research of both authors.  
\end{dedication}

\end{document}